\documentclass[11pt,a4paper]{amsart}
\usepackage{graphicx}
\usepackage{amsmath,amsfonts}
\usepackage{amsthm,amssymb,latexsym}
\usepackage[active]{srcltx}
\usepackage[usenames]{color}
\usepackage{tikz}
\usetikzlibrary{matrix,decorations.pathreplacing,calc}

\DeclareMathOperator{\rank}{rank}
\newcommand{\stirling}[2]{\genfrac{[}{]}{0pt}{}{#1}{#2}}

\newcommand{\bfs}{\boldsymbol}

\newtheorem{theorem}{Theorem}[section]
\newtheorem{corollary}[theorem]{Corollary}
\newtheorem{lemma}[theorem]{Lemma}

\newtheorem*{fact*}{Fact}
\newtheorem{proposition}[theorem]{Proposition}
\theoremstyle{definition}

\newtheorem{remark}[theorem]{Remark}
\numberwithin{equation}{section}
\newtheorem{algorithm}[theorem]{Algorithm}
\newcommand{\N}{\mathbb N}

\newcommand{\A}{\mathbb A}
\newcommand{\F}{\mathbb F}
\newcommand{\K}{\mathbb K}

\newcommand{\Pp}{\mathbb P}

\newcommand{\FF}{{\sf F}}

\newcommand{\fq}{\F_{\hskip-0.7mm q}}

\newcommand{\fqs}{\F_{\hskip-0.7mm q^s}}

\newcommand{\cfq}{\overline{\F}_{\hskip-0.7mm q}}

\def\ifm#1#2{\relax \ifmmode#1\else#2\fi}

\newcommand{\klk}    {\ifm {,\ldots,} {$,\ldots,$}}

\begin{document}
\title[Computation of rational solutions of underdetermined systems]{On the
computation of rational solutions of underdetermined systems over a
finite field}
\author[N. Gimenez]{Nardo Gim\'enez${}^{1}$}%

\author[G. Matera]{Guillermo Matera${}^{1,2}$}%

\author[M. P\'erez]{Mariana P\'erez${}^{2,3}$}

\author[M. Privitelli]{Melina Privitelli${}^{2,4}$}

\address{${}^{1}$Instituto del Desarrollo Humano,
Universidad Nacional de Gene\-ral Sarmiento, J.M. Guti\'errez 1150
(B1613GSX) Los Polvorines, Buenos Aires, Argentina}
\email{\{agimenez,\,gmatera\}@ungs.edu.ar}
\address{${}^{2}$ National Council of Science and Technology (CONICET),
Ar\-gentina}
\address{${}^{3}$
Universidad Nacional de Hurlingham, Instituto de Tecnolog\'ia e
Ingenier\'ia\\ Av. Gdor. Vergara 2222 (B1688GEZ), Villa Tesei,
Buenos Aires, Argentina} \email{mariana.perez@unahur.edu.ar}
\address{${}^{4}$Instituto de Ciencias,
Universidad Nacional de Gene\-ral Sarmiento, J.M. Guti\'errez 1150
(B1613GSX) Los Polvorines, Buenos Aires, Argentina}
\email{mprivite@ungs.edu.ar}

\thanks{The authors were partially supported by the grants UNGS 30/1146
and PICTO-UNAHUR-2019-00012.}%
\date{\today}
\maketitle

\begin{abstract}
We design and analyze an algorithm for computing solutions with
coefficients in a finite field $\fq$ of underdetermined systems
defined over $\fq$. The algorithm is based on reductions to
zero-dimensional searches. The searches are performed on ``vertical
strips'', namely parallel linear spaces of suitable dimension in a
given direction. Our results show that, on average, less than three
searches suffice to obtain a solution of the original system, with a
probability of success which grows exponentially with the number of
searches. The analysis of our algorithm relies on results on the
probability that the solution set (over the algebraic closure of
$\fq$) of a random system with coefficients in $\fq$ satisfies
certain geometric and algebraic properties which is of independent
interest.
\end{abstract}
%
%
\section{Introduction}\label{section: intro}
Let $\fq$ be the finite field of $q$ elements and let $\cfq$ denote
its algebraic closure. For $K=\fq$ or $K=\cfq$, we denote by
$K[X_1,\ldots,X_r]$ the ring of polynomials in the indeterminates
$X_1,\ldots,X_r$ and coefficients in $K$. For any $d\in\N$, we
denote by $\mathcal{F}_d$ the set of elements of
$\fq[X_1,\ldots,X_r]$ of degree at most $d$. Given $1<s<r$ and
polynomials $F_1,\ldots, F_s\in\mathcal{F}_d$, in this paper we are
concerned with the problem of finding a common $\fq$-rational
solution, namely a solution with coefficients in $\fq$, of the {\em
underdetermined} system $F_1=0,\ldots,F_s=0$.

For $s=1$ and $r=2$, namely for the problem of finding
$\fq$-rational points of plane curves, this is considered in
\cite{GaShSi03}. For this purpose, the authors propose a combination
of specialization and unidimensional search, which is called {\em
Search on a vertical strip} (SVS). If the plane curve under
consideration is defined by a polynomial $F\in\fq[X_1,X_2]$, the
idea consists of considering successive specializations $F(a,X_2)$,
where the values $a\in\fq$ are randomly chosen, until an
$\fq$-rational point of the plane curve $\{F=0\}$ is obtained in a
``vertical strip'' $\{a\}\times\fq$. As there are efficient
algorithms for finding $\fq$-rational zeros of univariate
polynomials in $\fq[X]$, the critical point consists of determining
how many random choices must be done to have a ``good'' probability
of finding a vertical strip $\{a\}\times\fq$ with an $\fq$-rational
point of the plane curve under consideration. In \cite{GaShSi03}
this is done using an explicit version of the Weil estimate on the
number of $\fq$-rational points of absolutely irreducible (i.e.,
irreducible over $\cfq$) plane curves defined over $\fq$.

The case $s=1$ and arbitrary $r$, namely the problem of finding
$\fq$-rational points on hypersurfaces, is considered in
\cite{Matera10} and \cite{MaPePr17}. Given an absolutely irreducible
$F\in\fq[X_1,\ldots,X_r]$, in \cite{Matera10} it is shown that $\deg
F$ random choices of $\bfs a\in\fq^{r-1}$ suffice to have
probability greater than 1/2 of reaching a ``vertical strip''
$\{\bfs a\}\times \fq\subset\fq^r$ with $\fq$-rational points of the
hypersurface $\{F=0\}$ under consideration. As ``most'' polynomials
$F\in\fq[X_1,\ldots,X_r]$ of a given degree are absolutely
irreducible, this analysis comprises most of the elements of
$\fq[X_1,\ldots,X_r]$ (see \cite{GaViZi13} for explicit estimates on
the number of absolutely irreducible elements of
$\fq[X_1,\ldots,X_r]$). On the other hand, in \cite{MaPePr17} there
is an average-case analysis which considers all the elements of
$\fq[X_1,\ldots,X_r]$ of a given degree, assuming a uniform
distribution of inputs. Such analysis shows that, on average, three
random choices of $\bfs a\in\fq^{r-1}$ suffice. This is further
reflected by the fact that the average-case complexity of the
corresponding procedure is asymptotically at most three times the
cost of computing a specialization $F(\bfs a,X_r)$ and finding the
$\fq$-rational roots of  $F(\bfs a,X_r)$.

Now, for the general case $1<s<r$, we may distinguish two different
approaches. On one hand, we have the so-called ``rewriting''
methods, particularly Gr\"obner basis methods and variants such as
XL and MXL, which have been widely used for solving polynomial
systems over finite fields. In \cite{ArFaImKaSu04} and
\cite{AlCiFaPe12} it was shown that XL and MXL methods can be
described essentially as variants of the F4 algorithm for computing
Gr\"obner bases \cite{Faugere99}. Therefore, we shall consider
exclusively Gr\"obner basis algorithms. In particular,
\cite{BeFaPe09} considers the problem of finding $\fq$-rational
solutions of underdetermined systems over finite fields. An
``hybrid'' approach is applied, which consists on a mix of partial
specialization and zero-dimensional solving using Gr\"obner bases.
The amounts of partial specialization and zero-dimensional solving
is determined by a heuristic analysis, which allows the authors to
experimentally break several multivariate crytographic schemes
(TRMS, UOV, and others) with parameters assumed to be secure until
then.

On the other hand, we have ``geometric'' methods, which compute a
suitable description of the variety $V$ over $\cfq$ defined by the
system under consideration (see, e.g., \cite{HuWo99} or
\cite{CaMa06a}). This description consists of one or several
hypersurfaces, which represent the image of components of $V$ under
suitable linear projections, together with some local inverses of
these linear projections. After some random trials, which may be
seen as zero-dimensional linear sections ---or specializations--- of
the input system, an $\fq$-rational point of one of these
hypersurfaces is obtained, which yields an $\fq$-rational solution
of the input system. In both \cite{HuWo99} and \cite{CaMa06a}, the
proposed number of random trials is proportional to the degree of
the hypersurfaces under consideration.

In this paper we propose a systematic treatment of these strategies
which combine partial specialization with zero-dimensional solving
over $\fq$, together with an average-case analysis of the algorithm
under consideration. As the algorithm we propose is reminiscent of
the ones of \cite{GaShSi03} and \cite{MaPePr17}, for bivariate and
$r$--variate polynomials respectively, we call it ``Search on
Vertical Strips'' (SVS for short). The algorithm is the following.

\begin{algorithm}\label{algo: basic scheme} ${}$

Input: polynomials $\bfs F_s:=(F_1,\ldots,F_s)\in\mathcal{F}_d^s$.

Output: either a zero $\bfs x\in\fq^r$ of $\bfs F_s$, or
``failure''.

Set $i:=1$ and ${\tt State}:={\tt Failure}$

While $1\leq i\leq r-s+1$ and ${\tt State}={\tt Failure}$ do
\begin{enumerate}
  \item[] Choose at random $\bfs a_i\in\fq^{r-s}\setminus\{\bfs a_1\klk \bfs a_{i-1}\}$
  \item[] Find a common $\fq$-rational zero of $\bfs F_s(\bfs a_i,X_{r-s+1},\ldots,X_r)$
  \item[] If $\bfs x\in\fq^s$ is such a zero, then set ${\tt State}:={\tt Success}$
  \item[] $i:=i+1$
\end{enumerate}

End While

If ${\tt State}={\tt Success}$, then return $(\bfs a_i,\bfs x)$;
else, return ``{\tt Failure}''.
\end{algorithm}

Ignoring the cost of random generation of elements of $\fq^{r-s}$,
at the $i$th step of the main loop we compute the vector of
coefficients of the polynomials $F_j(\bfs a_i,X_{r-s+1},\ldots,
X_r)$ for $1\le j\le s$. Since an element of $\mathcal{F}_d$ has
$D:=\binom{d+r}{r}$ coefficients, the number of arithmetic
operations in $\fq$ required to compute such a vector is
$\mathcal{O}^\sim(sD)$, where the notation $\mathcal{O}^\sim$
ignores logarithmic factors. Throughout this paper, all asymptotic
estimates are valid for fixed $s$ and $r$, with $d\ll q$ growing to
infinity. Then a common $\fq$-rational zero of $\bfs F_s(\bfs
a_i,X_{r-s+1},\ldots,X_r)$ is computed, provided that there is such
a zero. This can be done using one's favorite zero-dimensional
solver over $\fq$. It may be based on Gr\"obner bases as in
\cite{BeFaPe09}, or it may be a ``geometric'' solver as the ones of
\cite{HuWo99}, \cite{CaMa06a} or \cite{HoLe21}. As a solver based on
Gr\"obner bases computations as F5 \cite{Faugere02}, and the
``Kronecker'' solver of \cite{CaMa06a} or \cite{HoLe21}, seem to be
the most efficient from a complexity point of view, we shall express
our costs in terms of the cost of these solvers exclusively. Denote
by $\tau(d,s,q)$ the maximum number of arithmetic operations
required to compute a specialization $\bfs F_s(\bfs a_i,-)$ and
solve a zero-dimensional system $\bfs F_s(\bfs a_i,-)=\bfs 0$ as
above with any of these solvers. Given a choice $\underline{\bfs
a}:=(\bfs a_1\klk \bfs a_{r-s+1})$ for the vertical strips to be
considered, the whole procedure requires
$\mathcal{O}^\sim\big(C_{\underline{\bfs a}}(\bfs F_s)\cdot
\tau(s,d,q)\big)$ arithmetic operations in $\fq$, where
$C_{\underline{\bfs a}}(\bfs F_s)$ is the least value of $i$ for
which the zero-dimensional solver under consideration succeeds to
compute a common solution in $\fq^s$ of the system $\bfs F_s(\bfs
a_i,-)=\bfs 0$.

The average-case analysis of Algorithm \ref{algo: basic scheme}
shows that this scheme, in which zero-dimensional searches are
performed only on a limited number of vertical strips, is efficient
and has a good probability of success. This is due to the fact that,
as more vertical strips are considered, the probability of finding
$\fq$-rational solutions decreases exponentially, and therefore the
additional computational effort becomes progressively useless.

The average-case analysis of Algorithm \ref{algo: basic scheme} is
heavily based on the properties of a specialization of a random
system $\bfs F_s(\bfs a,-)=\bfs 0$ with $\bfs F_s\in\mathcal{F}_d^s$
and $\bfs a\in\fq^{r-s}$. More precisely, we show that such a
specialization is defined by a {\em reduced regular sequence} in
$\fq[X_{r-s+1},\ldots,X_r]$ with high probability (see Section
\ref{subsec: reduced reg sequences} for the definition of a reduced
regular sequence). This is crucial from the complexity point of
view, as both Gr\"obner basis methods and the Kronecker solver are
shown to behave well on systems satisfying this condition (see,
e.g., \cite{BaFaSa15} or \cite[Chapter 26]{BoChGiLeLeSaSc17} for
Gr\"obner bases and \cite{CaMa06a} or \cite{HoLe21} for the
Kronecker solver).
We have the following result (see Corollary \ref{coro: probability 1
specializ and red reg sequence}).
\begin{theorem}
\label{th: probability reduced reg sequence and abs irred - intro}
Let $\bfs a\in\fq^{r-s}$, and let $N$ be the number of $\bfs
F_s:=(F_1 \klk F_s) \in \mathcal{F}_d^s$ such that $\bfs F_s(\bfs
a,-):=\big(F_1(\bfs a,-),\ldots,F_s(\bfs a,-)\big)$ forms a reduced
regular sequence. If $q
>2d^s(d+1)^s$, then
$$1-\frac{2d^s(d+1)^s}{q} \leq \frac{N}{|\mathcal{F}_d^s|} \leq 1.$$
\end{theorem}

Then we analyze the number of partial specializations which are
necessary to reach to a vertical strip where the system under
consideration has an $\fq$-rational solution. It turns out that the
probability that a high number $h$ of specializations are required
decreases exponentially with $h$. More precisely, we have the
following result (see Theorems \ref{th: probability 1
specialization} and \ref{th: probability of h specializations} and
Corollaries \ref{coro: probability 1 specialization - asymptotic}
and \ref{coro: probability of h specializations - asymptotic} for
precise statements).
\begin{theorem}\label{th: probability of h specializations - intro}
For any $m\in\N$ denote $\mu_m:=\sum_{j=1}^m\frac{(-1)^{j-1}}{j!}$.
Let $C$ be the random variable which counts the number of
specializations required to obtain a vertical strip of a random
system $\bfs F_s=\bfs 0$ as above with an $\fq$-rational solution.
Suppose that $q^s>\max\{d,6\}$ and $1<h\le r-s+1$. We have
\begin{align*}
\big|P[C=h]-\mu_d(1-\mu_d)^{h-1}\big| &\le
\mbox{$\frac{1}{(d+1)!}$}\big(e^{h-1}+\mbox{$\frac{1}{2}$}\big)
+\mathcal{O}\big(\mbox{$\frac{1}{q}$}\big)\ \textrm{ for }d\textrm{
odd},\\[1ex]
\big|P[C=h]-\mu_{d+1}(1-\mu_{d+1})^{h-1}\big| &\le
\mbox{$\frac{1}{(d+1)!}$}\big(e^{h-1}+\mbox{$\frac{1}{2}$}\big)
+\mathcal{O}\big(\mbox{$\frac{1}{q}$}\big)\ \textrm{ for }d\textrm{
even},
\end{align*}
where $P$ denotes probability and $e$ is the basis of the natural
logarithm.
\end{theorem}

Observe that $\mu_d\approx 1-e^{-1}= 0.6321\ldots$ for large $d$. We
remark that the quantity $\mu_d$ arises also in connection with a
classical combinatorial notion over finite fields, that of the value
set of univariate polynomials (cf. \cite{LiNi83}, \cite{MuPa13}). As
$\mu_d(1-\mu_d)\approx 0.2325\ldots$ for large $d>s$ and $q^s>d$, we
may paraphrase Theorem \ref{th: probability of h specializations -
intro} as saying that 2 specializations will suffice with high
probability to obtain a vertical strip with an $\fq$-rational
solution of the system under consideration. We observe that the
probabilistic algorithms of \cite{HuWo99}, \cite{CaMa06a} and
\cite{Matera10} propose a number of searches of order
$\mathcal{O}(d^s)$ to achieve a probability of success greater than
1/2, while that of \cite{BeFaPe09} proposes $q$ specializations. Our
result suggests that these analyzes are somewhat pessimistic.

Finally, we analyze the average-case complexity and probability of
success of Algorithm \ref{algo: basic scheme} (see Theorems \ref{th:
average-case compl} and \ref{th: probability success} for precise
statements).
\begin{theorem}
\label{th: average-case compl - intro} Let $h^*:={r-s+1}$. For
$q>2d^s(d+1)^s$ and $d>s$, the average--case complexity $E$ of
Algorithm \ref{algo: basic scheme} is bounded in the following way:
$$
E \leq \left\{\begin{array}{l} \tau(d,s,q) \Big(\mu_d^{-1}
+h^*(1-\mu_d)^{h^*}+\mbox{$\frac{3{h^*}e^{h^*}}{(d+1)!}$}+\mathcal{O}\big(\mbox{$\frac{r
(d+1)^{2r}}{q}$}\big)\Big)\ \textrm{ for }d\text{ odd},\\
\tau(d,s,q) \Big(\mu_{d+1}^{-1}
+h^*(1-\mu_{d+1})^{h^*}+\mbox{$\frac{3{h^*}e^{h^*}}{(d+1)!}$}+\mathcal{O}\big(\mbox{$\frac{r
(d+1)^{2r}}{q}$}\big)\Big)\ \textrm{ for }d\text{ even},
\end{array}\right.
$$
where $\tau(d,s,q)$ is the cost of the search in a vertical strip
and the constant underlying the $\mathcal{O}$-notation is
independent of $r$, $s$, $d$ and $q$. Further, the probability $P$
of failure of Algorithm \ref{algo: basic scheme} can be bounded as
follows:
\begin{align*}
\big|P-(1-\mu_d)^{h^*}\big| &\le \mbox{$\frac{e^{h^*}}{(d+1)!}$}
+\mathcal{O}(\mbox{$\frac{r(d+1)^{2r}}{q}$})\ \textrm{ for
}d\text{ odd},\\
\big|P-(1-\mu_{d+1})^{h^*}\big| &\le \mbox{$\frac{e^{h^*}}{(d+1)!}$}
+\mathcal{O}(\mbox{$\frac{r(d+1)^{2r}}{q}$})\ \textrm{ for }d\text{
even},\end{align*}
where the constant underlying the $\mathcal{O}$-notation is
independent of $r$, $s$, $d$ and $q$.
\end{theorem}

As $1/\mu_d\approx 1.58\ldots$ and $h^*(1-\mu_d)^{h^*}<0.76$, this
result suggests that, on average, at most $1/\mu_d+0.76\approx
2.34\ldots$ vertical strips must be searched to obtain an
$\fq$-rational zero of the polynomial system under consideration,
for large $d$ and $q$ satisfying the hypotheses of Theorem \ref{th:
average-case compl - intro}.

The paper is organized as follows. In Section \ref{section:
notation, notations} we briefly recall the notions and notations of
algebraic geometry and finite fields we use. Section \ref{section:
probability reg complete inters} is devoted to estimate the
probability that a specialization of a random system forms a reduced
regular sequence, proving Theorem \ref{th: probability reduced reg
sequence and abs irred - intro}. In Section \ref{section: number of
specializ for having fq-solutions} we estimate the expected number
of vertical strips to be searched to have an $\fq$-rational solution
of the system under consideration, showing Theorem \ref{th:
probability of h specializations - intro}. Finally, in Section
\ref{section: analysis of SVS algorithm} we apply the results of
Sections \ref{section: probability reg complete inters} and
\ref{section: number of specializ for having fq-solutions} to
establish Theorem \ref{th: average-case compl - intro}.
%
%
\section{Preliminaries}\label{section: notation, notations}
We use standard notions and notations of commutative algebra and
algebraic geometry as can be found in, e.g., \cite{Harris92},
\cite{Kunz85} or \cite{Shafarevich94}.

Let $\K$ be any of the fields $\fq$ or $\cfq$. We denote by $\A^r$
the $r$--dimensional affine space $\cfq{\!}^r$ and by $\Pp^r$ the
$r$--dimensional projective space over $\cfq$. 
%
By a {\em projective variety defined over} $K$ (or a projective
$K$--variety for short) we mean a subset $V\subset \Pp^r$ of common
zeros of homogeneous polynomials $F_1,\ldots, F_m \in K[X_0 ,\ldots,
X_r]$. Correspondingly, an {\em affine variety of $\A^r$ defined
over} $K$ (or an affine $K$--variety) is the set of common zeros in
$\A^r$ of polynomials $F_1,\ldots, F_{m} \in
K[X_1,\ldots, X_r]$. 
We shall frequently denote by $V(\bfs F_m)=V(F_1\klk F_m)$ or
$\{\bfs F_m=\bfs 0\}=\{F_1=0\klk F_m=0\}$ the affine or projective
$K$--variety consisting of the common zeros of the polynomials $\bfs
F_m:=(F_1\klk F_m)$.

In what follows, unless otherwise stated, all results referring to
varieties in general should be understood as valid for both
projective and affine varieties. A $K$--variety $V$ is $K$--{\em
irreducible} if it cannot be expressed as a finite union of proper
$K$--subvarieties of $V$. Further, $V$ is {\em absolutely
irreducible} if it is $\cfq$--irreducible as a $\cfq$--variety. Any
$K$--variety $V$ can be expressed as an irredundant union
$V=\mathcal{C}_1\cup \cdots\cup\mathcal{C}_s$ of irreducible
(absolutely irreducible) $K$--varieties, unique up to reordering,
which are called the {\em irreducible} ({\em absolutely
irreducible}) $K$--{\em components} of $V$.


For a $K$--variety $V$ contained in $\Pp^r$ or $\A^r$, we denote by
$I(V)$ its {\em defining ideal}, namely the set of polynomials of
$K[X_0,\ldots, X_r]$, or of $K[X_1,\ldots, X_r]$, vanishing on $V$.
The {\em coordinate ring} $K[V]$ of $V$ is the quotient ring
$K[X_0,\ldots,X_r]/I(V)$ or $K[X_1,\ldots,X_r]/I(V)$. The {\em
dimension} $\dim V$ of $V$ is the length $n$ of the longest chain
$V_0\varsubsetneq V_1 \varsubsetneq\cdots \varsubsetneq V_n$ of
nonempty irreducible $K$--varieties contained in $V$. We say that
$V$ has {\em pure dimension} $n$ (or simply it is {\em
equidimensional}) if all the irreducible $K$--components of $V$ are
of dimension $n$.

A $K$--variety in $\Pp^r$ or $\A^r$ of pure dimension $r-1$ is
called a $K$--{\em hypersurface}. A $K$--hypersurface in $\Pp^r$ (or
$\A^r$) is the set of zeros of a single nonzero polynomial of
$K[X_0\klk X_r]$ (or of $K[X_1\klk X_r]$).
%
%
%

The {\em degree} $\deg V$ of an irreducible $K$-variety $V$ is the
maximum number of points lying in the intersection of $V$ with a
linear space $L$ of codimension $\dim V$, for which $V\cap L$ is a
finite set. More generally, following \cite{Heintz83} (see also
\cite{Fulton84}), if $V=\mathcal{C}_1\cup\cdots\cup \mathcal{C}_s$
is the decomposition of $V$ into irreducible $K$--components, we
define the degree of $V$ as
$$\deg V:=\sum_{i=1}^s\deg \mathcal{C}_i.$$
%
%
We shall use the following {\em B\'ezout inequality} (see
\cite{Heintz83}, \cite{Fulton84}, \cite{Vogel84}): if $V$ and $W$
are $K$--varieties of the same ambient space, then
\begin{equation}\label{eq: Bezout}
\deg (V\cap W)\le \deg V \cdot \deg W.
\end{equation}

Let $V\subset\A^r$ be a $K$--variety and $I(V)\subset K[X_1,\ldots,
X_r]$ its defining ideal. Let $\bfs x$ be a point of $V$. The {\em
dimension} $\dim_{\bfs x}V$ {\em of} $V$ {\em at} $\bfs x$ is the
maximum of the dimensions of the irreducible $K$--components of $V$
that contain $\bfs x$. If $I(V)=(F_1,\ldots, F_m)$, the {\em tangent
space} $T_{\bfs x}V$ {\em to $V$ at $\bfs x$} is the kernel of the
Jacobian matrix $(\partial F_i/\partial X_j)_{1\le i\le m,1\le j\le
r}(\bfs x)$ of the polynomials $F_1,\ldots, F_m$ with respect to
$X_1,\ldots, X_r$ at $\bfs x$. We have (see, e.g., \cite[page
94]{Shafarevich94})
$$\dim T_{\bfs x}V\ge \dim_{\bfs x}V.$$
The point $\bfs x$ is {\em regular} if $\dim T_{\bfs x}V=\dim_xV$.
Otherwise, the point $x$ is called {\em singular}. The set of
singular points of $V$ is the {\em singular locus}
$\mathrm{Sing}(V)$ of $V$; a variety is called {\em nonsingular} if
its singular locus is empty. For a projective variety, the concepts
of tangent space, regular and singular point can be defined by
considering an affine neighborhood of the point under consideration.
%
%

Let $V$ and $W$ be irreducible affine $K$--varieties of the same
dimension and $f:V\to W$ a regular map for which
$\overline{f(V)}=W$, where $\overline{f(V)}$ is the closure of
$f(V)$ with respect to the Zariski topology of $W$. Such a map is
called {\em dominant}. Then $f$ induces a ring extension
$K[W]\hookrightarrow K[V]$ by composition with $f$, and thus a field
extension $K(W)\hookrightarrow K(V)$ of the fraction fields $K(W)$
and $K(V)$ of $K[V]$ and $K[W]$ respectively. We define the {\em
degree} $\deg f$ of $f$ as the degree $[K(V):K(W)]$ of the (finite)
field extension $K(W)\hookrightarrow K(W)$. For $\bfs y\in W$ and
$\bfs x\in f^{-1}(\bfs y)$, we say that $f$ is {\em unramified} at
$\bfs x$ if the differential mapping $d_{\bfs x}f \colon T_{\bfs
x}V\to T_{\bfs y}W$ is injective. 
%
%
\subsection{Reduced regular sequences}
\label{subsec: reduced reg sequences} A \emph{set-theoretic complete
intersection} is a variety $V(F_1 \klk F_s) \subseteq \A^r$ defined
by $s\le r$ polynomials $F_1 \klk F_s\in K[X_1 \klk X_r]$ which is
of pure dimension $r-s$. If $s\le r+1$ homogeneous polynomials $F_1
\klk F_s$ in $K[X_0 \klk X_r]$ define a projective $K$-variety
$V(F_1 \klk F_s)\subseteq \mathbb{P}^r$ which is of pure dimension
$r-s$ (the case $r-s=-1$ meaning  that $V(F_1,\ldots,F_s)$ is the
empty set), then the variety is called a \emph{set-theoretic
complete intersection}. Elements $F_1 \klk F_s\in K[X_1 \klk X_r]$
form a \emph{regular sequence} if $(F_1 \klk F_s)$ is a proper ideal
of $K[X_1 \klk X_r]$, $F_1$ is nonzero and, for each $2\le i \le s$,
$F_i$ is neither zero nor a zero divisor in $K[X_1 \klk X_r]/(F_1
\klk F_{i-1})$. If in addition $(F_1 \klk F_i)$ is a radical ideal
of $K[X_1 \klk X_r]$ for $1\le i \le s$, then we say that $F_1 \klk
F_s$ form a \emph{reduced regular sequence}. 
\subsection{Rational points}
We denote by $\A^r(\fq)$ the $n$--dimensional $\fq$--vector space
$\fq^r$ and by $\Pp^r(\fq)$ the set of lines of the
$(r+1)$--dimensional $\fq$--vector space $\fq^{r+1}$. For a
projective variety $V\subset\Pp^r$ or an affine variety
$V\subset\A^r$, we denote by $V(\fq)$ the set of $\fq$--rational
points of $V$, namely $V(\fq):=V\cap \Pp^r(\fq)$ in the projective
case and $V(\fq):=V\cap \A^r(\fq)$ in the affine case. For an affine
variety $V$ of dimension $n$ and degree $\delta$ we have the upper
bound (see, e.g., \cite[Lemma 2.1]{CaMa06})
\begin{equation}\label{eq: upper bound -- affine gral}
   |V(\fq)|\leq \delta q^n.
\end{equation}
%
%
\section{Geometric properties satisfied by ``most'' systems}
\label{section: probability reg complete inters}
In the sequel, we fix $1< s< r$ and $d\ge 2$. Denote by
$\mathcal{F}_d:=\mathcal{F}_{r,d}$ the set of nonzero polynomials in
$\fq[X_1, \dots, X_r]$ of degree at most $d$ and
$\mathcal{F}_d^s:=\mathcal{F}_d\times \cdots \times \mathcal{F}_d$
($s$ times). Given any $F\in\fq[X_1 \klk X_r]$ of degree $d$, we
write $\mathrm{coeffs}(F)$ for the corresponding $\binom{d +
r}{r}$-tuple of coefficients (with respect to a given monomial
order). Further, for any $s$-tuple $\bfs F_s:=(F_1 \klk F_s)\in
\mathcal{F}_d^s$ we denote $\mathrm{coeffs}(\bfs
F_s):=(\mathrm{coeffs}(F_1) \klk \mathrm{coeffs}(F_s))$.

The probability of success of our algorithm relies heavily on the
properties of a specialization $\bfs F_s(\bfs a,-)$ of a ``random''
system $\bfs F_s=\bfs 0$, with $\bfs F_s\in\mathcal{F}_d^s$ and
$\bfs a\in\fq^{r-s}$. For this reason, we devote this section to
analyze the probability that a specialization $\bfs F_s(\bfs
a,-):=(F_1(\bfs a,-), \dots, F_s(\bfs a,-))$ of a random $\bfs
F_s\in\mathcal{F}_d^s$ satisfies certain critical properties for the
analysis of our algorithm. In particular, we shall be interested in
the following property:
\begin{center}
$({\sf H})$\qquad $\bfs F_s(\bfs a,-)$ forms a reduced regular
sequence in $\cfq[X_{r-s+1} \klk X_r]$.
\end{center}
In this section we show that most specializations $\bfs F_s(\bfs
a,-)$ satisfy condition $({\sf H})$. More precisely, we obtain a
lower bound close to 1 for the probability that, for a random choice
of $\bfs F_s:=(F_1, \dots, F_s)$ in $\mathcal{F}^s_d$ and $\bfs
a\in\fq^{r-s}$, the specialization $\bfs F_s(\bfs a,-)$ satisfies
condition $({\sf H})$.

We start with the following result.
\begin{lemma}\label{lemma: Bezout implies H}
Let $K$ be any field and let $G_1, \ldots, G_s$ be  $s\le n$
polynomials in $K[X_1, \ldots, X_n]$ with $\deg G_i \le d$ for $1\le
i \le s$. If $V(G_1, \ldots, G_s)\subset\mathbb{A}^n(\overline{K})$
has degree $d^s$, then $G_1, \ldots, G_s$ form a reduced regular
sequence.
\end{lemma}
\begin{proof}
Denote $\bfs G_i:=(G_1,\ldots,G_i)$ for $1\le i\le s$. By the
B\'ezout inequality \eqref{eq: Bezout},
$$\deg V(\bfs G_i)\le d\cdot
\deg V(\bfs G_{i-1})$$
for $1\le i\le s$. As a consequence, the hypothesis $\deg V(\bfs
G_s)=d^s$ implies that $\deg V(\bfs G_i)=d^i$ for $1\le i\le s$.
Further, by considering the sequence $G_i, G_1, \ldots, G_{i-1},
G_{i+1}, \ldots, G_s$, the previous argument shows that $\deg
V(G_i)=d$ for $1\le i\le s$. Since $\deg G_i\le d$, it follows that
$(G_i)$ is the vanishing ideal of the hypersurface $V(G_i)$ for
$1\le i\le s$.

Now suppose that $G_i$ is a zero divisor modulo $\bfs G_{i-1}$.
Assume further that $i\ge 2$ is the minimum index with this
property. Denote by $C_1,\ldots, C_t$ the irreducible components of
$V\big(\bfs G_{i-1}\big)$, which are all of pure dimension $n-i+1$.
By definition, we have $\sum_{j=1}^t\deg C_j=d^{i-1}$. As $G_i$ is a
zero divisor modulo $\bfs G_{i-1}$, it vanishes identically over an
irreducible component, say $C_1$. As a consequence,
\begin{align*}
\deg  V(\bfs G_i)&\le\sum_{j=1}^t\deg C_j\cap V(G_i)\\
&\le \deg C_1+\sum_{j=2}^t\deg C_j\cdot d< \sum_{j=1}^t\deg C_j\cdot
d=d^i,
\end{align*}
contradicting the fact that $\deg  V(\bfs G_i)=d^i$. We conclude
that $G_1,\ldots,G_s$ form a regular sequence.

We have already shown that $(G_i)$ is a radical ideal for $1\le i\le
s$. Now, let $2\le i \le s$ and assume inductively that the ideal
$(G_1, \ldots, G_{i-1})$ is radical. As $G_1, \ldots, G_{i}$ form a
regular sequence, $V(\bfs G_{i-1})$ and $V(G_i)$ are varieties of
pure dimension that intersect properly. On the other hand, we have
$$
\deg V(\bfs G_{i-1})\cdot \deg V(G_i) = \deg V(\bfs
G_i)=\sum_{Z}\deg Z,
$$
where the sum runs over all irreducible components $Z$ of $V(\bfs
G_i)$. By, e.g., \cite[Corollary 18.4]{Harris92} we deduce that
$V(\bfs G_{i-1})$ and $V(G_i)$ intersect transversely at a general
point $p$ of any component $Z$ of $V(\bfs G_i)$. Since, by the
inductive hypothesis, $G_1, \ldots, G_{i-1}$ generate the vanishing
ideal of $V(\bfs G_{i-1})$  and  $G_i$ generates the vanishing ideal
of $V(G_i)$, the fact that $V(\bfs G_{i-1})$ and $V(G_i)$ intersect
transversely at $p$ implies that the Jacobian matrix of $G_1,
\ldots, G_i$ has rank equal to $i$ at $p$. As the ideal $(G_1,
\ldots, G_i)$ has codimension $i$, by \cite[Theorem
18.15]{Eisenbud95} we conclude that $G_1, \ldots, G_{i}$ generate a
radical ideal. This concludes the proof of the lemma.
\end{proof}

Now we obtain an hypersurface containing all the $\bfs F_s$ for
which a specialization does not satisfy condition $({\sf H})$.
\begin{proposition} \label{prop: condition red reg sequence of specialization}
There exists a nonzero $\bfs P_S \in \fq[\mathrm{coeffs}(\bfs F_s)]$
with
$$\deg \bfs P_S \leq 2d^s(d+1)^s,$$
such that for any $\bfs F_s:=(F_1 \klk F_s) \in
\mathcal{F}_d^s(\cfq)$ with $\bfs P_S(\mathrm{coeffs}(\bfs F_s))
\neq 0$, the specialization $\bfs F_s(\bfs a,-)$ satisfies condition
$({\sf H})$.
\end{proposition}
\begin{proof}
Consider the incidence variety
$$W_{\bfs a}:=\{(\bfs F_s,\bfs x)\in\mathcal{F}_d^s(\cfq)\times \A^s:
\bfs F_s(\bfs a,\bfs x)=0\}.$$
Let ${\sf F}_1, \ldots, {\sf F}_s$ be formal polynomials of degree
$d$ in the variables $X_1, \ldots, X_r$. It is easy to see that
${\sf F}_1(\bfs a,X_{r-s+1},\ldots,X_r),\ldots, {\sf F}_s(\bfs a,
X_{r-s+1},\ldots, X_r)$ form a regular sequence of
$\fq[\mathrm{coeffs}({\sf \bfs F}_s), X_{r-s+1},\ldots,X_r]$. This
implies that $W_{\bfs a}$ is of pure dimension
$\dim\mathcal{F}_d^s(\cfq)$. Further, we may express $W_{\bfs a}$ as
the set of solutions of the following system:
$$\mathrm{coeff}_{\bfs 0}({\sf F}_i)=
-({\sf F}_i-\mathrm{coeff}_{\bfs 0}({\sf F}_i))(\bfs
a,X_{r-s+1},\ldots,X_r)\quad(1\le i\le s),$$
where $\mathrm{coeff}_{\bfs 0}({\sf F}_i)$ denotes the coefficient
of the monomial of degree zero of ${\sf F}_i$. We conclude $W_{\bfs
a}$ may be seen as the image of a regular mapping defined on the
absolutely irreducible variety $\mathcal{F}_d^s(\cfq)\times \A^s$,
and therefore it is absolutely irreducible.

Denote by $\pi_{\bfs a}:W_{\bfs a}\to\mathcal{F}_d^s(\cfq)$ its
projection on the first coordinate. It is easy to see that there
exists a zero-dimensional fiber. Indeed, let $f\in\fq[T]$ be a monic
irreducible polynomial of degree $d$, and let $\mathcal{S}:=V(f)$ be
the set of zeros of $f$. If $\bfs F_s:=\big(f(X_{r-s+1}),
\ldots,f(X_r)\big)$, then $\pi_{\bfs a}^{-1}(\bfs F_s)=\{\bfs
F_s\}\times \mathcal{S}^s$. Then the theorem on the dimension of
fibers (see, e.g., \cite[Theorem 8.13]{Cutkosky18}) implies that
$\pi_{\bfs a}$ is dominant.

Let $J(\bfs F_s(\bfs a,-))$ denote the Jacobian of $F_i(\bfs
a,X_{r-s+1},\ldots,X_r)$ $(1\le i\le s)$ with respect to
$X_{r-s+1},\ldots,X_r$. A straightforward computation shows that
$\pi_{\bfs a}$ is unramified at $(\bfs F_s,\bfs x)\in W_{\bfs a}$ if
$J(\bfs F_s(\bfs a,-))(\bfs x)\not=0$. For $F_i:=f(X_{r-s+i})$
$(1\le i \le s)$ as above, $J(\bfs F_s(\bfs
a,-))=\prod_{i=1}^sf'(X_{r-s+i})$.
This shows that  $J(\bfs F_s(\bfs a, -))(\bfs x)$ does not vanish
identically on $W_{\bfs a}$. Taking into account that $W_{\bfs a}$
is absolutely irreducible we conclude that the set of points at
which $\pi_{\bfs a}$ is unramified contains  a nonempty Zariski open
subset of $W_{\bfs a}$. \cite[Proposition T.8]{Iversen73} shows that
the field extension
$\cfq(\mathcal{F}_d^s(\cfq))\hookrightarrow\cfq(W_{\bfs a})$ is
separable, and \cite[Proposition 1]{Heintz83} proves that
$\#\pi_{\bfs a}^{-1}(\bfs F_s)\le\deg\pi_{\bfs a}$ for any finite
fiber $\pi_{\bfs a}^{-1}(\bfs F_s)$, with equality in a nonempty
Zariski open subset of $\mathcal{F}_d^s(\cfq)$. We call any fiber
unramified if it belongs to this Zariski open set, and ramified
otherwise.

Observe that any zero-dimensional fiber of $\pi_{\bfs a}$ consists
of at most $d^s$ points, due to the fact that it is defined by $s$
polynomials of degree at most $d$ and \eqref{eq: Bezout}. In
particular, for $F_i:=f(X_{r-s+i})$ $(1\le i \le s)$ as above, the
corresponding fiber has precisely $d^s$ points. This proves that
$\deg\pi_{\bfs a}=d^s$.

Now
\cite[Proposition 3]{Schost03} and its proof allows to describe more
precisely an open subset for the condition $\#\pi_{\bfs a}^{-1}(\bfs
F_s)=\deg\pi_{\bfs a}=d^s$. Indeed, denote by $G_i:={\sf F}_i(\bfs
a, X_{r-s+1}, \ldots, X_r)$ $(1\le i \le s)$ the polynomials of
$\cfq[\mathrm{coeffs}({\sf \bfs F}_s), X_{r-s+1}, \ldots, X_r]$
which define the incidence variety $W_{\bfs a} =
V_{\mathcal{F}_d^s(\cfq)\times \A^s}(G_1, \ldots, G_s)$, and write
$J(\bfs G_s)\in \cfq[\mathrm{coeff}({\sf \bfs F}_s), X_{r-s+1},
\ldots, X_r]$ for the Jacobian of $G_1, \ldots, G_s$ with respect to
$X_{r-s+1},\ldots,X_r$. Since $J(\bfs G_s)(\bfs F_s, \bfs x)=J(\bfs
F_s(\bfs a, -))(\bfs x)$ for any $(\bfs F_s, \bfs x)\in W_{\bfs a}$,
it follows by the previous discussion that $J(\bfs G_s)$ does not
vanish identically on $W_{\bfs a}$.

Let $\mathcal{V}:=V_{\mathcal{F}_d^s(\cfq)\times \A^s}\big((G_1,
\ldots, G_s):J(\bfs G_s)^{\infty}\big)$ be the nonempty subvariety
of $W_{\bfs a}$ defined by the saturated ideal $(G_1, \ldots,
G_s):J(\bfs G_s)^{\infty}$. Denote $\pi_{\bfs a}|_{\mathcal{V}}$ the
restriction of $\pi_{\bfs a}$ to $\mathcal{V}$. By the proof of
\cite[Proposition 3]{Schost03} we deduce that exists a nonzero
polynomial $\bfs P_S \in \fq[\mathrm{coeffs}(\bfs F_s)]$ of total
degree at most $2\deg (\pi_{\bfs a}|_{\mathcal{V}})\deg \mathcal{V}$
with the following property: for any $\bfs F_s\in
\mathcal{F}_d^s(\cfq)$ with $\bfs P_S(\bfs F_s)\not=0$, we have that
$\pi_{\bfs a}^{-1}(\bfs F_s)\cap \mathcal{V}$ is a finite set
containing $\deg (\pi_{\bfs a}|_{\mathcal{V}})$ points. Now, as
$W_{\bfs a}$ is irreducible and $J(\bfs G_s)$ does not vanish
identically on $W_{\bfs a}$, $W_{\bfs a}\setminus
V_{\mathcal{F}_d^s(\cfq)\times \A^s}(J(\bfs G_s))$ is an open dense
subset of $W_{\bfs a}$. Thus
$$\mathcal{V}=V_{\mathcal{F}_d^s(\cfq)\times \A^s}\big((G_1, \ldots, G_s):J(\bfs G_s)^{\infty}\big)=
\overline{W_{\bfs a}\setminus V_{\mathcal{F}_d^s(\cfq)\times
\A^s}(J(\bfs G_s))}=W_{\bfs a}.$$
We conclude that for any $\bfs F_s\in \mathcal{F}_d^s(\cfq)$, the
condition  $\bfs P_S(\bfs F_s)\not=0$ implies that the fiber
$\pi_{\bfs a}^{-1}(\bfs F_s)$ is finite and contains $\deg
(\pi_{\bfs a}|_{\mathcal{V}})=\deg \pi_{\bfs a}=d^s$ points.
 Further, observe that
$W_{\bfs a}$ is defined by $s$ polynomials of degree $d+1$. As a
consequence, by \eqref{eq: Bezout} we conclude that $\deg W_{\bfs
a}\le (d+1)^s$. It follows that
$$\deg \bfs P_S\le 2d^s(d+1)^s.$$

Finally, for each $\bfs F_s\in \mathcal{F}_d^s(\cfq)$ with $\bfs
P_S(\bfs F_s)\not=0$, we have that $\pi_{\bfs a}^{-1}(\bfs
F_s)=\{\bfs F_s\}\times V(\bfs F_s(\bfs a,-))$ is zero-dimensional
of degree $d^s$. Therefore, Lemma \ref{lemma: Bezout implies H}
implies that $\bfs F_s(\bfs a,-)$ satisfies condition ({\sf H}).
\end{proof}

Finally, we estimate the probability that a random specialization of
a random $\bfs F_s$ satisfies condition $({\sf H})$. For this
purpose, we consider the random variable $C_{\sf H}:=\fq^{r-s}\times
\mathcal{F}^s_d\rightarrow \{1, \infty\}$ defined in the following
way:
$$
C_{\sf H}(\bfs a,\bfs F_s):=\left\{
\begin{array}{rl}
1,&\!\!\mbox{if $\bfs F_s(\bfs a,-)$ satisfies ({\sf
H})};\\
\infty, &\!\! \hbox{otherwise}.
\end{array}
\right.
$$
We consider the set $\fq^{r-s}\times \mathcal{F}^s_d$ endowed with
the uniform probability $P_1:=P_{1,r,s,d}$ and study the probability
of
the set $\{C_{\sf H}=1\}$. 
%
%
%
%
%
%
\begin{corollary} \label{coro: probability 1 specializ and red reg sequence}
For $q >2d^s(d+1)^s$ and $S_{\sf H}:=\{C_{\sf H}=1\}$, we have
$$1-\frac{2d^s(d+1)^s}{q} \leq P_1(S_{\sf H}) \leq 1.$$
\end{corollary}
\begin{proof}
The upper bound being obvious, we prove the lower bound. For $\bfs
a\in\fq^{r-s}$, let $N$ be the number of $\bfs F_s \in
\mathcal{F}^s_d$ such that $\bfs F_s(\bfs a,-)$ satisfies condition
$({\sf H})$, and let $M$ be the number of $\bfs F_s\in
\mathcal{F}_d^s$ such that $\bfs P_S(\mathrm{coeffs}(\bfs F_s)) =
0$, where $\bfs P_S$ is the polynomial of Proposition \ref{prop:
condition red reg sequence of specialization}. Let $e:=2d^s(d+1)^s$
and $D:=s{d+r\choose r}$.
According to \eqref{eq: upper bound -- affine gral},
$$
\frac{N}{|\mathcal{F}^s_d|} \ge\frac{|\mathcal{F}^s_d|-M
}{|\mathcal{F}^s_d|}
    \geq  \frac{q^D-e q^{D-1}}{q^D}
    =1-\frac{e}{q}.
$$
%
%
This proves the corollary.
\end{proof}

%
%
\section{The number of specializations to have an $\fq$-rational solution}
\label{section: number of specializ for having fq-solutions}
For $1<s<r$ and $d \geq 2$, 
%
given $\bfs F_s:=(F_1, \dots, F_s)$ of $\mathcal{F}^s_{d}$, we are
interested in finding an $\fq$-rational solution of the system $\bfs
F_s=\bfs 0$. For this purpose, we consider Algorithm \ref{algo:
basic scheme}, which performs searches for $\fq$-rational solutions
on the (hopefully) zero-dimensional systems arising from a number of
specializations of $\bfs F_s$. More precisely, for each choice of
$\bfs a\in \fq^{r-s}$, we shall search for an $\fq$-rational
solution of the system $\bfs F_s=\bfs 0$ in the ``vertical strip''
$\{\bfs a\}\times \fq^{s}$ determined by $\bfs a$. Thus our
algorithm mainly consists of ``searches on vertical strips'' (SVS
for short).

In Section \ref{section: probability reg complete inters} we
estimate the probability that a specialization of the system $\bfs
F_s=\bfs 0$ has ``good'' properties from the computational point of
view. In this section we study the probability that $h$ random
specializations must be considered until a vertical strip with
$\fq$-rational solutions of the input system is attained.
%
%
\subsection{The probability of having $\fq$-rational solutions in
the first specialization}
\label{subsec: prob of zeros in the first specializ}
%
To analyze the probability of success of our algorithm in the first
search, we introduce the random variable $C_1:=\fq^{r-s}\times
\mathcal{F}^s_d\rightarrow \{1, \infty\}$ defined in the following
way:
$$
C_1(\bfs a,\bfs F_s):=\left\{
\begin{array}{rl}
1,&\!\!\!\!\mbox{if $\bfs F_s(\bfs a,-)$ has a zero in }\fq^s; \\
\infty,\!\!\!\! & \hbox{otherwise.}
\end{array}
\right.
$$
We consider the set $\fq^{r-s}\times \mathcal{F}^s_d$ endowed with
the uniform probability $P_1:=P_{1,s,r,d}$ and study the probability
of the set $\{C_1=1\}$. Explicitly, we aim to estimate the
probability $P_1(S_1)$, for
$$S_1:=\big\{(\bfs a,\bfs F_s) \in \fq^{r-s} \times \mathcal{F}^s_d:
\bfs F_s(\bfs a,-)\mbox{ has a zero in }\fq^s\big\}.$$

We shall need the following technical lemma.
\begin{lemma}\label{lemma: vandermonde 1 search}
For $s\le d+1$ and $\bfs\alpha_1, \dots, \bfs\alpha_s\in\A^r$
pairwise distinct, consider the $(s\times {d+r\choose r})$-matrix
${\sf A}_{s,d}:=(\bfs\alpha_h^{\bfs i})_{1\le h \le s; |\bfs i|\le
d}$, where $\bfs i$ runs over all exponents $\bfs i:=(i_1, \dots,
i_r)\in \mathbb{Z}^{r}_{\geq 0} $ in some ordering and $|\bfs
i|:=i_1+\cdots+i_s$. Then ${\sf A}_{s,d}$ has maximal rank $s$.
\end{lemma}
\begin{proof}
Let $\Lambda_1, \dots, \Lambda_r$ be new indeterminates over $\cfq$
and $\mathcal{L}:=\Lambda_1X_1 +\cdots+ \Lambda_rX_r$. For any $\bfs
c:=(c_1,\dots,c_r)\in \A^r$ we write $\ell_{\bfs c}:=c_1X_1 +\dots +
c_rX_r$. Since the points $\bfs\alpha_1, \dots, \bfs\alpha_s$ are
pairwise distinct,
$$P:=\prod_{1\le j< k\le s}\big(\mathcal{L}(\bfs\alpha_j)
-\mathcal{L}(\bfs\alpha_k)\big)$$
is a nonzero polynomial in $\cfq[\Lambda_1, \dots, \Lambda_r]$. It
follows that there exists $\bfs \lambda:=(\lambda_1, \dots,
\lambda_r)\in \A^r$ such that $P(\bfs\lambda)\neq 0$. This means
that
$\ell_{\bfs\lambda}(\bfs\alpha_j)\not=\ell_{\bfs\lambda}(\bfs\alpha_k)$
for $1\le j<k\le s$. We may also assume without loss of generality
that $X_1,\ldots,X_{r-1},\ell_{\bfs\lambda}$ are $\cfq$-linearly
independent.

Write $\bfs\alpha_h:=(\alpha_{h,1}, \dots, \alpha_{h,r})$ for $1\le
h \le s$. Observe that the $h$th row of ${\sf A}_{s,d}$ consists of
the specialization of the monomial basis $\{\bfs X^{\bfs
i}:=X_1^{i_1}\cdots X_r^{i_r}:|\bfs i|\le d\}$ of
$\mathcal{F}_{d,r}(\cfq):=\{F\in \cfq[X_1,\ldots,X_r]:\deg F\le d\}$
at $\bfs\alpha_h$. Let
$\bfs\ell:=(X_1,\ldots,X_{r-1},\ell_{\bfs\lambda})$. Since $\{\bfs
\ell^{\bfs i}:=X_1^{i_1}\cdots
X_{r-1}^{i_{r-1}}\ell_{\bfs\lambda}^{i_r}:|\bfs i|\le d\}$ is
another basis of $\mathcal{F}_{d,r}(\cfq)$ as $\cfq$-vector space,
it turns out that $$\rank {\sf A}_{s,d}=\rank \big((\bfs\ell(\bfs
\alpha_h))^{\bfs i}\big)_{1\le h \le s; |\bfs i|\le d}.$$
We may therefore assume without loss of generality that
$\alpha_{h,r}\neq \alpha_{k,r}$ for $1\le h<k\le s$.

Consider the submatrix ${\sf V}_{s,d}:=(\alpha_{h,r}^{e})_{1\le h
\le s; 0\le e \le d}$ of ${\sf A}_{s,d}$. Observe that ${\sf
V}_{s,d}$ is a Vandermonde-like matrix of size $s\times (d+1)$.
Since $s\le d+1$, it is well known that our assumption implies that
${\sf V}_{s,d}$ has rank $s$, and so does ${\sf A}_{s,d}$.
\end{proof}

Now we are able to estimate the probability $P_1(S_1)$. From now on,
for any $m\in\N$ we shall use the notations
%
$$s_m:=\sum_{j=1}^m(-1)^{j-1}\binom{q^s}{j}q^{-sj},\quad
t_m:=\binom{q^s}{m}q^{-sm}.$$
%
\begin{theorem} \label{th: probability 1 specialization}
For $s\le d+1$ and $q^s>d$, we have
\begin{align*}
s_d-t_{d+1} \le P_1(S_1)\le  s_d\quad \textrm{for }d \textrm{ odd},
\qquad s_d   \le P_1(S_1)\le s_d+t_{d+1}\quad \textrm{for }d
\textrm{ even}.
\end{align*}
\end{theorem}
\begin{proof}
For any $\bfs F_s\in \mathcal{F}^s_d$, we denote by $VS(\bfs F_s)$
the set of vertical strips where the system $\bfs F_s=\bfs 0$ has an
$\fq$-rational solution and by $NS(\bfs F_s)$ its cardinality, that
is,
$$
VS(\bfs F_s):=\{\bfs a \in \fq^{r-s}: (\exists \, \bfs x\in
\fq^{s})\,  \bfs F_s(\bfs a, \bfs x)=\bfs 0\}, \quad NS(\bfs
F_s):=|VS(\bfs F_s)|.
$$
It is easy to see that  $S_1=\bigcup_{\bfs F_s\in
\mathcal{F}^s_d}VS(\bfs F_s)\times \{\bfs F_s\}$. Since this is a
disjoint union of $\fq^{r-s}\times \mathcal{F}^s_d$, it follows that
$$
P_1(S_1)=\frac{1}{q^{r-s}|\mathcal{F}^s_d|}\sum_{\bfs F_s\in
\mathcal{F}^s_d}NS(\bfs F_s).
$$
Fix $\bfs F_s\in \mathcal{F}^s_d$. Observe that
$$
VS(\bfs F_s)=\bigcup_{\bfs x\in \fq^s}\{\bfs a\in \fq^{r-s}: \bfs
F_s(\bfs a, \bfs x)=\bfs 0\}.
$$
%
Assume that $d$ is odd. By the Bonferroni inequalities we deduce
that
    \begin{align*}
      NS(\bfs F_s) &\geq  \sum_{j=1}^{d+1}(-1)^{j-1}\sum_{\mathcal{X}_j\subset \fq^s}|\{\bfs a\in \fq^{r-s}:
        (\forall \bfs x\in \mathcal{X}_j)\, \bfs F_s(\bfs a, \bfs x)=\bfs 0\}|, \\
      NS(\bfs F_s) &\leq \sum_{j=1}^{d}(-1)^{j-1}\sum_{\mathcal{X}_j\subset \fq^s}|\{\bfs a\in \fq^{r-s}:
        (\forall \bfs x\in \mathcal{X}_j)\, \bfs F_s(\bfs a, \bfs x)=\bfs 0\}| .
    \end{align*}
For $1\le j \le d+1$, denote
\begin{align*}
\mathcal{N}_j&=\frac{1}{q^{r-s}|\mathcal{F}^s_d|}\sum_{\bfs F_s\in
\mathcal{F}^s_d}\sum_{\mathcal{X}_j \subset \fq^s}\big|\{\bfs a\in
\fq^{r-s}: (\forall \bfs x\in
\mathcal{X}_j)\, \bfs F_s(\bfs a, \bfs x)=\bfs 0\}\big|\\
&=\frac{1}{q^{r-s}|\mathcal{F}_{d}|^s}\sum_{\mathcal{X}_j \subset
\fq^s}\sum_{\bfs a\in \fq^{r-s}} |\{\bfs F_s\in \mathcal{F}_{d}^s:
(\forall \bfs x\in \mathcal{X}_j)\, \bfs F_s(\bfs a, \bfs x)=\bfs
0\}|.\end{align*}
We conclude that
%
%
\begin{equation}\label{eq: bonferroni first specialization}
\sum_{j=1}^{d+1}(-1)^{j-1}\mathcal{N}_j \le P(S_1)\le
\sum_{j=1}^{d}(-1)^{j-1}\mathcal{N}_j.
\end{equation}

Next we obtain an explicit expression for each $\mathcal{N}_j$. We
have a direct-product decomposition of vector spaces
    $$
\{\bfs F_s\in \mathcal{F}_{d}^s: (\forall \bfs x\in \mathcal{X}_j)\,
\bfs F_s(\bfs a, \bfs x)=\bfs 0\}=\prod_{i=1}^s\{F_i\in
\mathcal{F}_{d}: (\forall \bfs x\in \mathcal{X}_j)\, F_i(\bfs a,
\bfs x)=0\}.
    $$
Further, all the factors in the product in the right-hand side are
isomorphic $\fq$-vector spaces. As a consequence, we may write
\begin{align*}
\mathcal{N}_j&=\frac{1}{q^{r-s}}\sum_{\mathcal{X}_j \subset
\fq^s}\sum_{\bfs a\in
\fq^{r-s}}\Bigg(\frac{1}{|\mathcal{F}_{d}|}\big|\{F\in
\mathcal{F}_{d}: (\forall \bfs x\in \mathcal{X}_j)\, F(\bfs a, \bfs
x)=\bfs 0\}\big|\Bigg)^s.\end{align*}
Let $j\le d+1$ and $\bfs a\in \fq^s$ be fixed. The set of equalities
$F(\bfs a, \bfs x)=0$ $(\bfs x\in \mathcal{X}_j)$ constitute $j$
linear conditions on the coefficients of $F$, which can be expressed
by a $j\times{d+r\choose d}$-matrix of rank $j$ by Lemma \ref{lemma:
vandermonde 1 search}. We conclude that
\begin{align*}
\mathcal{N}_j&=\sum_{\mathcal{X}_j \subset
\fq^s}\Bigg(\frac{1}{|\mathcal{F}_{d}|}\big|\{F\in \mathcal{F}_{d}:
(\forall \bfs x\in \mathcal{X}_j)\, F(\bfs a,
\bfs x)=\bfs 0\}\big|\Bigg)^s\\
&=\sum_{\mathcal{X}_j \subset \fq^s}\Bigg(\frac{q^{\dim\mathcal
F_{r,d}-j}}{|\mathcal{F}_{d}|}\Bigg)^s
=\binom{q^s}{j}q^{-sj}.\end{align*}
By \eqref{eq: bonferroni first specialization} it follows that
$$\sum_{j=1}^{d+1}(-1)^{j-1}\binom{q^s}{j}q^{-sj} \le P(S_1)
\le \sum_{j=1}^{d}(-1)^{j-1}\binom{q^s}{j}q^{-sj},$$
which proves the statement for $d$ odd.

The statement for $d$ even follows by a similar argument {\em
mutatis mutandis}.
\end{proof}

We now discuss the asymptotic behavior of the probability
$P_1(S_1)=P_1[C_1=1]$.
\begin{corollary} \label{coro: probability 1 specialization - asymptotic}
    For $s\le d+1$ and $q^s>d$, we have
    \begin{align*}
      \mu_{d+1} - \frac{2}{q^s}   &
      \le P_1(S_1) \le \mu_d + \frac{2}{q^s}\quad \textrm{for }d \textrm{ odd},\\
       \mu_{d} - \frac{2}{q^s}   &
\le P_1[C_1=1] \le \mu_{d+1} + \frac{2}{q^s} \quad \textrm{for }d
\textrm{ even}.
    \end{align*}
\end{corollary}
\begin{proof}
For positive integers $k,j$ with $k\le j$, we denote by
$\stirling{j}{k}$ the unsigned Stirling number of the first kind,
namely the number of permutations of $j$ elements with  $k$ disjoint
cycles. The following properties of Stirling numbers are well-known
(see, e.g., \cite[Section A.8]{FlSe09})
    $$\stirling{j}{j}=1,\ \ \stirling{j}{{j-1}}=\binom{j}{2},\ \ \sum_{k=0}^j\stirling{j}{k}=j!.$$
     We shall also use the following well-known identity (see, e.g., \cite[(6.13)]{GrKnPa94}):
    \begin{equation}\label{eq: binomial in terms of stirling numbers}
    \binom {q^s}{j}=\sum_{k=0}^j\frac{(-1)^{j-k}}{j!}\stirling{j}{k}q^{sk}.
    \end{equation}
According to \eqref{eq: binomial in terms of stirling numbers}, for
any positive integer $m$ we have that
    \begin{align*}
    s_m=\sum_{j=1}^{m}(-1)^{j-1}\binom{q^s}{j}q^{-sj}&= \sum_{j=1}^{m}(-1)^{j-1}\sum_{k=0}^j
    \frac{(-1)^{j-k}}{j!}\stirling{j}{k}q^{s(k-j)}
    \\
    &= \sum_{j=1}^m \frac{(-1)^{j-1}}{j!}\stirling{j}{j}+ \sum_{j=1}^m
    \frac{(-1)^j}{j!}\stirling{j}{j-1}q^{-s}\\&\quad+
    \sum_{j=1}^m\sum_{k=0}^{j-2}
    \frac{(-1)^{k-1}}{j!}\stirling{j}{k}q^{s(k-j)}.
    \end{align*}
    It follows that
    $$
    s_m=\mu_m+\frac{1}{q^s}\sum_{j=1}^m
    \frac{(-1)^j}{j!}\binom{j}{2}- \sum_{j=1}^m\sum_{k=0}^{j-2}
    \frac{(-1)^k}{j!}\stirling{j}{k}q^{s(k-j)}.
    $$
    As a consequence, for $m > 2$ we obtain
\begin{equation}\label{eq: s_m asymptotic}
|s_m-\mu_m|\le \frac{1}{q^s}\bigg|\sum_{j=1}^m
    \frac{(-1)^j}{j!}\binom{j}{2}\bigg|+ \sum_{j=1}^m\sum_{k=0}^{j-2}
    \frac{1}{j!}\stirling{j}{k}\frac{1}{q^{2s}}
\le \frac{1}{4q^s} +\frac{m}{q^{2s}}.\end{equation}
For $m=2$, this inequality is obtained by a direct calculation.
Combining this for $m=d$ and $m=d+1$ with Theorem \ref{th:
probability 1 specialization} we readily deduce the statement of the
corollary.
\end{proof}

%
%
\subsection{The probability that $h$ specializations are required}
Now we study the probability that multiple specializations are
required to attain a vertical strip with $\fq$-rational solutions of
the input system $\bfs F_s=\bfs 0$. More precisely, for $2\le h\le
r-s+1$, we analyze the probability that a random choice $\bfs
a_1,\ldots,\bfs a_h\in\fq^{r-s}$ yields systems $\bfs F_s(\bfs
a_i,-)=\bfs 0$ for $1\le i\le h-1$ with no $\fq$-rational solutions,
such that $\bfs F_s(\bfs a_h,-)=\bfs 0$ has $\fq$-rational
solutions.

Let $\bfs{\underline{a}}:=(\bfs a_1,\ldots,\bfs a_h)$ be such that
$\bfs a_i\not=\bfs a_j$ for $i\not=j$ and
$$S_{h,\bfs{\underline{a}}}^*:=\{\bfs F_s\in\mathcal{F}_{d}^s:
Z(\bfs F_s(\bfs a_i,-))=0\ (1\le i\le h-1),\ Z(\bfs F_s(\bfs
a_h,-))>0\},$$
where $Z(\bfs G_s)$ denotes the number of zeros in $\fq^s$ of $\bfs
G_s\in \fq[X_{r-s+1},\ldots,X_r]^s$. We may express
$S_{h,\bfs{\underline{a}}}^*$ in the following way:
$$S_{h,\bfs{\underline{a}}}^*=S_{1,\bfs a_h}^*\setminus
\bigcup_{j=1}^{h-1}S_{2,(\bfs a_j,\bfs a_h)},$$
where $S_{2,(\bfs a_j,\bfs a_h)}:=\{\bfs F_s\in\mathcal{F}_{d}^s:
Z(\bfs F_s(\bfs a_j,-))>0,\ Z(\bfs F_s(\bfs a_h,-))>0\}$. Therefore,
by the inclusion-exclusion principle, we obtain
\begin{align}
|S_{h,\bfs{\underline{a}}}^*|&=|S_{1,\bfs a_h}^*|-\Bigg|
\bigcup_{j=1}^{h-1}S_{2,(\bfs a_j,\bfs a_h)}\Bigg|\notag\\
&=|S_{1,\bfs a_h}^*|+\sum_{k=1}^{h-1}(-1)^k\sum_{1\le
i_1<\cdots<i_k< h}\big|S_{2,(\bfs a_{i_1},\bfs a_h)}\cap\cdots\cap
S_{2,(\bfs
a_{i_k},\bfs a_h)}\big|\notag\\
&=|S_{1,\bfs a_h}^*|+\sum_{k=1}^{h-1}(-1)^k\sum_{1\le
i_1<\cdots<i_k< h}\big|S_{k+1,(\bfs a_{i_1},\ldots,\bfs a_{i_k},\bfs
a_h)}\big|, \label{eq: S_h in terms of S_k}
\end{align}
where $S_{k,(\bfs b_1,\ldots,\bfs b_k)}:=\{\bfs
F_s\in\mathcal{F}_{d}^s: Z(\bfs F_s(\bfs b_j,-))>0\ for\ 1\le j\le
k\}$.

To estimate $|S_{h,\bfs{\underline{a}}}^*|$, we establish a
condition on $\bfs{\underline{a}}$ which implies that
$\big|S_{k+1,(\bfs a_{i_1},\ldots,\bfs a_{i_k},\bfs a_h)}\big|$ has
a predictable behavior. Then we estimate $\big|S_{k+1,(\bfs
a_{i_1},\ldots,\bfs a_{i_k},\bfs a_h)}\big|$ for $1\le k\le h-1$,
which will lead us to an estimate for
$|S_{h,\bfs{\underline{a}}}^*|$. Finally, we shall be able to
estimate the probability that $h$ random choices $\bfs
a_1,\ldots,\bfs a_h$ are required to have an $\fq$-rational solution
of the system under consideration.
%
%
\subsubsection{A condition on $\bfs{\underline{a}}$}
We start with the condition on $\bfs{\underline{a}}$ mentioned
above. Given positive integers $d\ge j_1\ge j_2\ge\cdots\ge j_h\ge
1$ and sets $\mathcal{X}_{j_i}\subset \fq^s$ of cardinality $j_i$
for $1\le i\le h$, we shall be interested in the number of $\bfs
F_s:=(F_1,\ldots,F_s)\in\mathcal{F}_{d}^s$ satisfying
$$\bfs F_s(\bfs a_i,\bfs x)=0 \textrm{ for all}\ \bfs
x\in\mathcal{X}_{j_i}\textrm{ and }1\le i\le h.
$$
We have the following result.
\begin{proposition}\label{prop: ranks generalized Vandermonde}
Denote $\bfs a_i:=(a_{i,1},\ldots,a_{i,r-s})$ for $1\le i\le h$ and
let
$${\sf M}:=\left(
  \begin{array}{cccc}
    1 & a_{1,1} & \ldots & a_{1,h-1} \\
    1 & a_{2,1} & \ldots & a_{2,h-1} \\
    \vdots & \vdots &  & \vdots \\
    1 & a_{h,1} & \ldots & a_{h,h-1} \\
  \end{array}
\right).$$
If ${\sf M}$ is invertible, then
$$\left|\{\bfs F_s\in\mathcal{F}_{d}^s:\bfs F_s(\bfs a_i,\bfs x)=0\
\forall\bfs x\in\mathcal{X}_{j_i}\textrm{ and }1\le i\le
h\}\right|=q^{s(\dim\mathcal{F}_{d}-(j_1+\cdots+j_h))}.$$
\end{proposition}
\begin{proof}
Observe that
\begin{align}
\{\bfs F_s\in\mathcal{F}_{d}^s:\bfs F_s(\bfs a_i,\bfs x)=0\
\forall\bfs x\in\mathcal{X}_{j_i}\textrm{ and }1\le i\le
h\}\qquad\notag\\\qquad=\prod_{i=1}^s \{F\in\mathcal{F}_{d}:F(\bfs
a_i,\bfs x)=0\ \forall\bfs x\in\mathcal{X}_{j_i}\textrm{ and }1\le
i\le h\}.\label{eq: product structure solution sets}\end{align}
As a consequence, we shall analyze the number of
$F\in\mathcal{F}_{d}$ satisfying
\begin{equation}\label{eq: condition Vandermonde}
F(\bfs a_i,\bfs x)=0\textrm{ for all }\bfs
x\in\mathcal{X}_{j_i}\textrm{ and }1\le i\le h.\end{equation}

We may consider \eqref{eq: condition Vandermonde} as a linear system
consisting on $j_1+\cdots+j_h$ equations on the ${d+r\choose r}$
coefficients of $F$, expressing $F$ in the monomial basis $\{\bfs
X^{\bfs\alpha}:=X_1^{\alpha_1}\cdots X_r^{\alpha_r}:|\bfs\alpha|\le
d\}$. We claim that the matrix of this system has full rank.

Up to a change of the monomial basis under consideration, we may
assume without loss of generality that $X_r$ separates the points of
$\mathcal{X}_{j_i}$ for each $i$. To prove the claim, it suffices to
show that the $(j_1+\cdots+j_h)\times(j_1+\cdots+j_h)$-submatrix
${\sf A}$ consisting of the $j_1+\cdots+j_h$ rows of this matrix and
the columns corresponding to the monomials
$\{1,X_r,\ldots,X_r^{j_1-1},
X_{i-1},X_{i-1}X_r,\ldots,X_{i-1}X_r^{j_i-1}\ (2\le i\le h)\}$ has
full rank $j_1+\cdots+j_h$.

Denote
$${\sf V}(\mathcal{X}_k,l):=
\left(
  \begin{array}{cccc}
    1 & x_1 & \ldots & x_1^{l-1} \\
    \vdots & \vdots &  & \vdots \\
    1 & x_k & \ldots & x_k^{l-1} \\
  \end{array}
\right),
$$
where $x_1,\ldots,x_k$ are the $r$th coordinates of the elements of
the set $\mathcal{X}_k$ (of cardinality $k$). Then the matrix ${\sf
A}$ can be expressed in the following way:
$${\sf A}=
\left(
  \begin{array}{cccc}
    {\sf V}(\mathcal{X}_{j_1},j_1) & a_{1,1}{\sf V}(\mathcal{X}_{j_1},j_2) & \ldots & a_{1,h-1}{\sf V}(\mathcal{X}_{j_1},j_h) \\
    {\sf V}(\mathcal{X}_{j_2},j_1) & a_{2,1}{\sf V}(\mathcal{X}_{j_2},j_2) & \ldots & a_{2,h-1}{\sf V}(\mathcal{X}_{j_2},j_h) \\
    \vdots & \vdots &  & \vdots \\
    {\sf V}(\mathcal{X}_{j_h},j_1) & a_{h,1}{\sf V}(\mathcal{X}_{j_h},j_2) & \ldots & a_{h,h-1}{\sf V}(\mathcal{X}_{j_h},j_h) \\
  \end{array}
\right).
$$
For positive integers $k,l$, let $\bfs 1_{k,l}$ be the $(k\times
l)$-matrix having 1 on the diagonal elements and 0 elsewhere. Then
the matrix ${\sf A}$ admits the following factorization:
$${\sf A}=
\left(
  \begin{array}{cccc}
   {\sf V}(\mathcal{X}_{j_1},j_1) &  \\
     & {\sf V}(\mathcal{X}_{j_2},j_1) \\
     &  & \ddots  \\
     &  &   & {\sf V}(\mathcal{X}_{j_h},j_1) \\
  \end{array}
\right)\cdot \left(
  \begin{array}{cccc}
    \bfs 1_{j_1,j_1} & a_{1,1}\bfs 1_{j_1,j_2} & \ldots & a_{1,h-1}\bfs 1_{j_1,j_h} \\
    \bfs 1_{j_1,j_1} & a_{2,1}\bfs 1_{j_1,j_2} & \ldots & a_{2,h-1}\bfs 1_{j_1,j_h} \\
    \vdots & \vdots &  & \vdots \\
    \bfs 1_{j_1,j_1} & a_{h,1}\bfs 1_{j_1,j_2} & \ldots & a_{h,h-1}\bfs 1_{j_1,j_h} \\
  \end{array}
\right).
$$

The fact that $j_1\ge j_i$ and $\mathcal{X}_{j_i}$ consists of $j_i$
pairwise-distinct elements implies that ${\sf
V}(\mathcal{X}_{j_i},j_1)$ has full rank $j_i$ for $1\le i\le h$. It
follows that the left-hand matrix in the factorization of ${\sf A}$
has full rank, and thus
$$\rank{\sf A}=
\rank \left(
  \begin{array}{cccc}
    \bfs 1_{j_1,j_1} & a_{1,1}\bfs 1_{j_1,j_2} & \ldots & a_{1,h-1}\bfs 1_{j_1,j_h} \\
    \bfs 1_{j_1,j_1} & a_{2,1}\bfs 1_{j_1,j_2} & \ldots & a_{2,h-1}\bfs 1_{j_1,j_h} \\
    \vdots & \vdots &  & \vdots \\
    \bfs 1_{j_1,j_1} & a_{h,1}\bfs 1_{j_1,j_2} & \ldots & a_{h,h-1}\bfs 1_{j_1,j_h} \\
  \end{array}
\right).$$

Up to a row permutation in the matrix ${\sf M}$ of the statement of
the proposition, we may assume that we can perform Gauss elimination
without permutations to obtain a row echelon form of ${\sf M}$. By
hypothesis, such a row echelon form is upper triangular and
invertible. Similarly, we can perform a process of block Gauss
elimination on the right-hand matrix in the factorization of ${\sf
A}$ {\em mutatis mutandis}. As a consequence, we obtain a block
matrix which is block upper triangular, with invertible full-rank
blocks on the diagonal. It follows that this matrix has full rank
$j_1+\cdots+j_j$, and thus so does ${\sf A}$. This finishes the
proof of our claim.

As a consequence of the claim, the $j_1+\cdots+j_h$ equations of
\eqref{eq: condition Vandermonde} are linearly independent, which
implies that the set of solutions of \eqref{eq: condition
Vandermonde} has codimension $j_1+\cdots+j_h$. It follows that
\eqref{eq: condition Vandermonde} has
$q^{\dim\mathcal{F}_{d}-(j_1+\cdots+j_h)}$ solutions in
$\mathcal{F}_{d}$. Finally, taking into account \eqref{eq: product
structure solution sets} we readily deduce the statement of the
proposition.
\end{proof}

\begin{remark}\label{rem: cardinality elements bad vandermonde rank}
With notations as in Proposition \ref{prop: ranks generalized
Vandermonde}, there are
$$(q^h-q)(q^h-q^2)\cdots(q^h-q^{h-1})q^{h(r-s-h+1)}=
(q^{h-1}-1)(q^{h-2}-1)\cdots(q-1)q^{h(r-s-\frac{h-1}{2})}$$
elements $\bfs a_1,\ldots,\bfs a_h$ satisfying the condition of
Proposition \ref{prop: ranks generalized Vandermonde}. Indeed, it is
easy to see that there are $(q^h-q)(q^h-q^2)\cdots(q^h-q^{h-1})$
possible choices of the $h\times (h-1)$ right submatrix of ${\sf M}$
such that ${\sf M}$ is invertible. Then the factor $q^{h(r-s-h+1)}$
takes into account that the last $r-s-(h-1)$ coordinates of $\bfs
a_1,\ldots,\bfs a_h$ can be arbitrarily chosen.
\end{remark}
%
%
\subsubsection{Estimates for $|S_{k,\bfs{\underline{a}}}|$}
Let $\bfs{\underline{a}}:=(\bfs a_1,\ldots,\bfs
a_h)\in(\fq^{r-s})^h$ be such that $(\bfs a_1,\ldots,\bfs a_k)$
satisfies the hypothesis of Proposition \ref{prop: ranks generalized
Vandermonde} for $2\le k\le h$. As asserted before, we aim to
estimate the number of $\bfs F_s$ with $Z(\bfs F_s(\bfs a_k,-))=0$
for $1\le k\le h-1$ and $Z(\bfs F_s(\bfs a_h,-))>0$. According to
\eqref{eq: S_h in terms of S_k}, this can be reduced to estimate the
number of elements of the set
$$S_{k,\bfs{\underline{a}}}:=\{\bfs
F_s\in\mathcal{F}_{d}^s: Z(\bfs F_s(\bfs a_j,-))>0\ for\ 1\le j\le
k\}$$
for $2\le k\le r-s+1$. Assume that $d$ is odd. Since
$$S_{k,\bfs{\underline{a}}}=\bigcup_{\bfs x^1\in\fq^s}\cdots
\bigcup_{\bfs x^k\in\fq^s}\{\bfs F_s\in\mathcal{F}_{d}^s: \bfs
F_s(\bfs a_j,\bfs x^j)=0\ for\ 1\le j\le k\},$$
by the Bonferroni inequalities, we have
\begin{align*}
\sum_{j_1=1}^{d+1}(-1)^{j_1-1}\sum_{\mathcal{X}_{j_1}\subset\fq^s}A_{\mathcal{X}_{j_1}}
\le |S_{k,\bfs{\underline{a}}}|
\le\sum_{j_1=1}^d(-1)^{j_1-1}\sum_{\mathcal{X}_{j_1}\subset\fq^s}A_{\mathcal{X}_{j_1}},\hskip2cm\\
A_{\mathcal{X}_{j_1}}:=A_{\mathcal{X}_{j_1}}^{\bfs{\underline{a}}}:=\Bigg|\bigcup_{\bfs
x^2\in\fq^s}\cdots \bigcup_{\bfs x^k\in\fq^s}\{\bfs
F_s\in\mathcal{F}_{d}^s:\bfs F_s(\bfs a_1,\mathcal{X}_{j_1})=\bfs
F_s(\bfs a_j,\bfs x^j)=0\ for\ 2\le j\le k\}\Bigg|.\end{align*}
We may rewrite the previous inequalities in the following form:
\begin{equation}\label{eq: recursive bound k=1}
\sum_{\stackrel{\scriptstyle j_1=1}{j_1\textrm{
odd}}}^{d+1}\sum_{\mathcal{X}_{j_1}\subset\fq^s}A_{\mathcal{X}_{j_1}}-
\sum_{\stackrel{\scriptstyle j_1=1}{j_1\textrm{
even}}}^{d+1}\sum_{\mathcal{X}_{j_1}\subset\fq^s}A_{\mathcal{X}_{j_1}}
\le |S_{k,\bfs{\underline{a}}}| \le\sum_{\stackrel{\scriptstyle
j_1=1}{j_1\textrm{
odd}}}^d\sum_{\mathcal{X}_{j_1}\subset\fq^s}A_{\mathcal{X}_{j_1}}-
\sum_{\stackrel{\scriptstyle j_1=1}{j_1\textrm{
even}}}^d\sum_{\mathcal{X}_{j_1}\subset\fq^s}A_{\mathcal{X}_{j_1}}.
\end{equation}

Next, we apply the Bonferroni inequalities to each
$A_{\mathcal{X}_{j_1}}$. More precisely, we have
\begin{align}\notag
\sum_{\stackrel{\scriptstyle j_2=1}{j_2\textrm{
odd}}}^{d+1}\sum_{\mathcal{X}_{j_2}\subset\fq^s}A_{\mathcal{X}_{j_1},\mathcal{X}_{j_2}}-
\!\!\sum_{\stackrel{\scriptstyle j_2=1}{j_2\textrm{
even}}}^{d+1}\sum_{\mathcal{X}_{j_2}\subset\fq^s}A_{\mathcal{X}_{j_1},\mathcal{X}_{j_2}}\!
&\le  \\A_{\mathcal{X}_{j_1}}&\le\sum_{\stackrel{\scriptstyle
j_2=1}{j_2\textrm{
odd}}}^d\sum_{\mathcal{X}_{j_2}\subset\fq^s}A_{\mathcal{X}_{j_1},\mathcal{X}_{j_2}}-
\sum_{\stackrel{\scriptstyle j_2=1}{j_2\textrm{
even}}}^d\sum_{\mathcal{X}_{j_2}\subset\fq^s}A_{\mathcal{X}_{j_1},\mathcal{X}_{j_2}},
\label{eq: recursive bound k=2}
\end{align}
where
\begin{align*}
&A_{\mathcal{X}_{j_1},\mathcal{X}_{j_2}}:=\\\ &\Bigg|\bigcup_{\bfs
x^3\in\fq^s}\cdots \bigcup_{\bfs x^k\in\fq^s}\{\bfs
F_s\in\mathcal{F}_{d}^s:\bfs F_s(\bfs a_1,\mathcal{X}_{j_1})=\bfs
F_s(\bfs a_2,\mathcal{X}_{j_2})=\bfs F_s(\bfs a_j,\bfs x^j)=0\ for\
3\le j\le k\}\Bigg|.\end{align*}

Proceeding in this way, we obtain upper and lower bounds for
$|S_{k,\bfs{\underline{a}}}|$ in terms of signed double sums of
terms of the form
$$A_{\mathcal{X}_{j_1},\ldots,\mathcal{X}_{j_k}}:=\big|\{\bfs
F_s\in\mathcal{F}_{d}^s:\bfs F_s(\bfs
a_1,\mathcal{X}_{j_1})=\cdots=\bfs F_s(\bfs
a_k,\mathcal{X}_{j_k})=0\}\big|.$$
According to Proposition \ref{prop: ranks generalized Vandermonde},
$A_{\mathcal{X}_{j_1},\ldots,\mathcal{X}_{j_k}}
=q^{s\dim\mathcal{F}_d-s(j_1+\cdots+j_k)}$, depending only of the
cardinality of the sets $\mathcal{X}_{j_1},\ldots,\mathcal{X}_{j_k}$
under consideration. Therefore, rearranging sums we may express the
upper and lower bounds for $|S_{k,\bfs{\underline{a}}}|$ as $2^k$
signed sums over $j_1,\ldots,j_k$ of terms of the form
$$\sum_{\mathcal{X}_{j_1}\subset\fq^s}\cdots
\sum_{\mathcal{X}_{j_k}\subset\fq^s}A_{\mathcal{X}_{j_1},\ldots,\mathcal{X}_{j_k}}=
{q^s\choose j_1}\cdots{q^s\choose
j_k}q^{s\dim\mathcal{F}_d-s(j_1+\cdots+j_k)}.$$

Denote by $U_k$ and $L_k$ the expressions of the upper and lower
bounds for $\frac{|S_{k,\bfs{\underline{a}}}|}{|\mathcal{F}_d^s|}$
obtained in this way. According to \eqref{eq: recursive bound k=1},
for $k=1$ one has
$$L_1\le \frac{|S_{1,\bfs a_1}|}{|\mathcal{F}_d^s|}\le U_1,$$
where
\begin{align*}
U_1&:=\sum_{\stackrel{\scriptstyle j_1=1}{j_1\textrm{
odd}}}^d{q^s\choose j_1}q^{-sj_1}- \sum_{\stackrel{\scriptstyle
j_1=1}{j_1\textrm{ even}}}^d{q^s\choose
j_1}q^{-sj_1}=\sum_{j_1=1}^d(-1)^{j_1-1}{q^s\choose
j_1}q^{-sj_1}=:s_d,\\
L_1&:=\sum_{\stackrel{\scriptstyle j_1=1}{j_1\textrm{
odd}}}^{d+1}{q^s\choose j_1}q^{-sj_1}- \sum_{\stackrel{\scriptstyle
j_1=1}{j_1\textrm{ even}}}^{d+1}{q^s\choose j_1}q^{-sj_1}
=\sum_{j_1=1}^{d+1}(-1)^{j_1-1}{q^s\choose j_1}q^{-sj_1}=:s_{d+1}.
\end{align*}

For $m\in\N$, denote
$$s_{m,\textrm{even}}:=\sum_{\stackrel{\scriptstyle
i=1}{i\textrm{ even}}}^m{q^s\choose i}q^{-si},\quad
s_{m,\textrm{odd}}:=\sum_{\stackrel{\scriptstyle i=1}{i\textrm{
odd}}}^m{q^s\choose i}q^{-si}.$$
We have the following result.
\begin{lemma}\label{lemma: recursive relation UL}
For $2\le k\le r-s+1$ and $d$ odd, we have
\begin{align*}
U_k&=U_{k-1} \sum_{\stackrel{\scriptstyle i=1}{i\textrm{
odd}}}^d{q^s\choose i}q^{-si}- L_{k-1}\sum_{\stackrel{\scriptstyle
i=1}{i\textrm{ even}}}^d{q^s\choose i}q^{-si} =U_{k-1}
s_{d,\textrm{odd}}- L_{k-1}s_{d,\textrm{even}},\\
L_k&=L_{k-1} \sum_{\stackrel{\scriptstyle i=1}{i\textrm{
odd}}}^{d+1}{q^s\choose i}q^{-si}-
U_{k-1}\sum_{\stackrel{\scriptstyle i=1}{i\textrm{
even}}}^{d+1}{q^s\choose i}q^{-si}=L_{k-1} s_{d+1,\textrm{odd}}-
U_{k-1}s_{d+1,\textrm{even}}.
\end{align*}
On the other hand, for $d$ even,
\begin{align*}
U_k&=U_{k-1} \sum_{\stackrel{\scriptstyle i=1}{i\textrm{
odd}}}^{d+1}{q^s\choose i}q^{-si}-
L_{k-1}\sum_{\stackrel{\scriptstyle i=1}{i\textrm{
even}}}^{d+1}{q^s\choose i}q^{-si} =U_{k-1}
s_{d+1,\textrm{odd}}- L_{k-1}s_{d+1,\textrm{even}},\\
L_k&=L_{k-1} \sum_{\stackrel{\scriptstyle i=1}{i\textrm{
odd}}}^d{q^s\choose i}q^{-si}- U_{k-1}\sum_{\stackrel{\scriptstyle
i=1}{i\textrm{ even}}}^d{q^s\choose i}q^{-si}=L_{k-1}
s_{d,\textrm{odd}}- U_{k-1}s_{d,\textrm{even}}.
\end{align*}
\end{lemma}
\begin{proof}
Assume that $d$ is odd. Arguing by induction on $k$, for $k=2$ we
have
$A_{\mathcal{X}_{j_1},\mathcal{X}_{j_2}}=q^{s\dim\mathcal{F}_d-s(j_1+j_2)}=
|\mathcal{F}_d^s|q^{-s(j_1+j_2)}$. Therefore, by \eqref{eq:
recursive bound k=2} we have
\begin{align*}
q^{-sj_1}\Bigg(\sum_{\stackrel{\scriptstyle j_2=1}{j_2\textrm{
odd}}}^{d+1}{q^s\choose j_2}q^{-sj_2}-&
\!\!\sum_{\stackrel{\scriptstyle j_2=1}{j_2\textrm{
even}}}^{d+1}{q^s\choose j_2}q^{-sj_2}\Bigg)\! \\&\le
\frac{A_{\mathcal{X}_{j_1}}}{|\mathcal{F}_d^s|} \le
q^{-sj_1}\Bigg(\sum_{\stackrel{\scriptstyle j_2=1}{j_2\textrm{
odd}}}^d{q^s\choose j_2}q^{-sj_2}- \sum_{\stackrel{\scriptstyle
j_2=1}{j_2\textrm{ even}}}^d{q^s\choose j_2}q^{-sj_2}\Bigg).
\end{align*}
As a consequence, by the definition of $L_1$ and $U_1$ it follows
that
$$q^{-sj_1}L_1 \le
\frac{A_{\mathcal{X}_{j_1}}}{|\mathcal{F}_d^s|} \le q^{-sj_1}U_1.$$
Using these bounds in \eqref{eq: recursive bound k=1} we obtain
$$\sum_{\stackrel{\scriptstyle j_1=1}{j_1\textrm{
odd}}}^{d+1}{q^s\choose j_1}q^{-sj_1}L_1-
\!\!\!\sum_{\stackrel{\scriptstyle j_1=1}{j_1\textrm{
even}}}^{d+1}{q^s\choose j_1}q^{-sj_1}U_1 \le
\frac{|S_{k,\bfs{\underline{a}}}|}{|\mathcal{F}_d^s|}
\le\sum_{\stackrel{\scriptstyle j_1=1}{j_1\textrm{
odd}}}^d{q^s\choose j_1}q^{-sj_1}U_1-\!\!\!
\sum_{\stackrel{\scriptstyle j_1=1}{j_1\textrm{ even}}}^d{q^s\choose
j_1}q^{-sj_1}L_1.$$
This proves the assertion for $k=2$.

Now, for $k>2$, taking into account that all the conditions imposed
on the $\bfs F_s$ are linearly independent, and the definition of
$L_k$ and $U_k$, by \eqref{eq: recursive bound k=1} we see that
\begin{align*}
\frac{|S_{k,\bfs{\underline{a}}}|}{|\mathcal{F}_d^s|}
&\le\sum_{\stackrel{\scriptstyle j_1=1}{j_1\textrm{
odd}}}^d\sum_{\mathcal{X}_{j_1}\subset\fq^s}\frac{A_{\mathcal{X}_{j_1}}}{|\mathcal{F}_d^s|}-
\sum_{\stackrel{\scriptstyle j_1=1}{j_1\textrm{
even}}}^d\sum_{\mathcal{X}_{j_1}\subset\fq^s}\frac{A_{\mathcal{X}_{j_1}}}{|\mathcal{F}_d^s|}
\\
&\le\sum_{\stackrel{\scriptstyle j_1=1}{j_1\textrm{
odd}}}^d{q^s\choose j_1}q^{-sj_1}U_{k-1}-
\sum_{\stackrel{\scriptstyle j_1=1}{j_1\textrm{ even}}}^d{q^s\choose
j_1}q^{-sj_1}L_{k-1}=U_{k-1} s_{d,\textrm{odd}}-
L_{k-1}s_{d,\textrm{even}}.
\end{align*}
This proves the assertion on $U_k$. The assertion on $L_k$ is showed
by a similar argument.

Finally, for $d$ even the bounds are obtained arguing as above {\em
mutatis mutandis}.
\end{proof}

For $d$ odd, we may express the recursive relation between the
$U_k,L_k$ in the following form:
$$\left(
    \begin{array}{c}
      U_k \\
      L_k \\
    \end{array}
  \right)={\sf A}_d
\left(
    \begin{array}{c}
      U_{k-1} \\
      L_{k-1} \\
    \end{array}
  \right),\quad {\sf A}_d:=
\left(
  \begin{array}{cc}
   s_{d,\textrm{odd}} & -s_{d,\textrm{even}} \\
 \!\! -s_{d+1,\textrm{even}} & s_{d+1,\textrm{odd}} \\
  \end{array}
\right)=\left(
  \begin{array}{cc}
   s_{d,\textrm{odd}} & -s_{d,\textrm{even}} \\
 \!\! -s_{d+1,\textrm{even}} & s_{d,\textrm{odd}} \\
  \end{array}
\right).
$$
We observe that
$${\sf C}_d:=\left(
  \begin{array}{cc}
   s_{d+1,\textrm{odd}} & -s_{d+1,\textrm{even}} \\
 \!\! -s_{d+1,\textrm{even}} & s_{d+1,\textrm{odd}} \\
  \end{array}
\right)\le {\sf A}_d\le \left(
  \begin{array}{cc}
   s_{d,\textrm{odd}} & -s_{d,\textrm{even}} \\
 \!\! -s_{d,\textrm{even}} & s_{d,\textrm{odd}} \\
  \end{array}
\right)=:{\sf B}_d,$$
where the signs $\le$ are understood componentwise. It follows that
$$ {\sf C}_d^{k-1} \left(
    \begin{array}{c}
      s_d \\
    \!\! s_{d+1}\!\! \\
    \end{array}
  \right)\le \left(
    \begin{array}{c}
      U_k \\
      L_k \\
    \end{array}
  \right)
={\sf A}_d^{k-1} \left(
    \begin{array}{c}
      U_1 \\
      L_1 \\
    \end{array}
  \right)
= {\sf A}_d^{k-1} \left(
    \begin{array}{c}
      s_d \\
    \!\! s_{d+1}\!\! \\
    \end{array}
  \right)\le {\sf B}_d^{k-1} \left(
    \begin{array}{c}
      s_d \\
    \!\! s_{d+1}\!\! \\
    \end{array}
  \right).
$$

${\sf B}_d$ is a symmetric positive-definite matrix. By considering
its spectral decomposition, one readily deduces the following
expression for its $(k-1)$th power, which is expressed in terms of
its eigenvalues $s_d^+:=s_{d,\textrm{odd}}+s_{d,\textrm{even}}$ and
$s_d=s_{d,\textrm{odd}}-s_{d,\textrm{even}}$:
$${\sf B}_d^{k-1}=\left(
  \begin{array}{cc}
   \frac{1}{2}(s_d^+)^{k-1}+\frac{1}{2}s_d^{k-1} & -\frac{1}{2}(s_d^+)^{k-1}+\frac{1}{2}s_d^{k-1} \\
 \!\!-\frac{1}{2}(s_d^+)^{k-1}+\frac{1}{2}s_d^{k-1} & \frac{1}{2}(s_d^+)^{k-1}+\frac{1}{2}s_d^{k-1} \\
  \end{array}
\right).$$
We conclude that
\begin{align*}
U_k&\le \big(\mbox{$\frac{1}{2}$}(s_d^+)^{k-1}
+\mbox{$\frac{1}{2}$}s_d^{k-1}\big)s_d+\big(-\mbox{$\frac{1}{2}$}(s_d^+)^{k-1}
+\mbox{$\frac{1}{2}$}s_d^{k-1}\big)s_{d+1}\\[1ex]
&=
\mbox{$\frac{1}{2}$}(s_d^+)^{k-1}(s_d-s_{d+1})+\mbox{$\frac{1}{2}$}s_d^{k-1}(s_d+s_{d+1})\\[1ex]
&= s_d^k+\mbox{$\frac{1}{2}{q^s\choose
d+1}$}q^{-s(d+1)}\big((s_d^+)^{k-1}-s_d^{k-1}\big).
\end{align*}
Similarly,
\begin{align*}
L_k&\ge \mbox{$\frac{1}{2}$}(s_{d+1}^+)^{k-1}(s_{d+1}-s_d)
+\mbox{$\frac{1}{2}$}s_{d+1}^{k-1}(s_d+s_{d+1})\\[1ex]
&= s_{d+1}^k-\mbox{$\frac{1}{2}{q^s\choose
d+1}$}q^{-s(d+1)}\big((s_{d+1}^+)^{k-1}-s_{d+1}^{k-1}\big).
\end{align*}
Next we obtain simpler upper and lower bounds for $|S_{k,(\bfs
a_1,\ldots,\bfs a_k)}|$.
\begin{proposition}\label{prop: bound for S_k}
Recall the notations
$s_d^+:=s_{d,\textrm{odd}}+s_{d,\textrm{even}}$,
$s_d:=s_{d,\textrm{odd}}-s_{d,\textrm{even}}$ and
$t_{d+1}:={q^s\choose d+1}q^{-s(d+1)}$. For $2\le k\le r-s+1$, we
have
\begin{align*}
\bigg|\frac{|S_{k,(\bfs a_1,\ldots,\bfs a_k)}|}{|\mathcal{F}_d^s|}-
s_d^k\bigg|&\le\mbox{$\frac{t_{d+1}}{2}$}\big((s_{d+1}^+)^{k-1}+(2k-1)s_d^{k-1}\big)\
\textrm{ for \mbox{$d$} odd},\\[1ex]
%
\bigg|\frac{|S_{k,(\bfs a_1,\ldots,\bfs a_k)}|}{|\mathcal{F}_d^s|}-
s_{d+1}^k\bigg|&\le\mbox{$\frac{t_{d+1}}{2}$}\big((s_{d+1}^+)^{k-1}+(2k-1)s_{d+1}^{k-1}\big)\
\textrm{ for \mbox{$d$} even}.
\end{align*}
\end{proposition}
\begin{proof}
Suppose that $d$ is odd. By the lower bound for $L_k$ above, we have
\begin{align*}
\frac{|S_{k,(\bfs a_1,\ldots,\bfs a_k)}|}{|\mathcal{F}_d^s|} &\ge
s_d^k+s_{d+1}^k-s_d^k-\mbox{$\frac{t_{d+1}}{2}$}\big((s_{d+1}^+)^{k-1}-s_{d+1}^{k-1}\big)\\
&\ge s_d^k-t_{d+1}(s_{d+1}^{k-1}+\cdots+s_d^{k-1})
-\mbox{$\frac{t_{d+1}}{2}$}\big((s_{d+1}^+)^{k-1}-s_{d+1}^{k-1}\big)\\[1ex]
&=
s_d^k-\mbox{$\frac{t_{d+1}}{2}$}\big((s_{d+1}^+)^{k-1}+s_{d+1}^{k-1}
+2s_{d+1}^{k-2}s_d+\cdots+2s_d^{k-1}\big)\\[1ex]
&\ge
s_d^k-\mbox{$\frac{t_{d+1}}{2}$}\big((s_{d+1}^+)^{k-1}+(2k-1)s_d^{k-1}\big).
\end{align*}
On the other hand,
$$
\frac{|S_{k,(\bfs a_1,\ldots,\bfs a_k)}|}{|\mathcal{F}_d^s|} \le
s_d^k+\mbox{$\frac{t_{d+1}}{2}$}\big((s_d^+)^{k-1}-s_d^{k-1}\big)
\le
s_d^k+\mbox{$\frac{t_{d+1}}{2}$}\big((s_{d+1}^+)^{k-1}+(2k-1)s_d^{k-1}\big).
$$
For $d$ even, the estimate follows similarly.
\end{proof}
%
%
\subsubsection{An estimate for $|S_{h,\bfs{\underline{a}}}^*|$}
Now we are able to estimate the probability that $h$ random choices
are made until we reach a vertical strip with $\fq$-rational
solutions of the system under consideration.

Recall that, given $\bfs{\underline{a}}:=(\bfs a_1,\ldots,\bfs a_h)$
which satisfies the hypothesis of Proposition \ref{prop: ranks
generalized Vandermonde}, we aim to estimate the probability of the
set
\begin{equation}\label{eq: def S_h^*}
S_{h,\bfs{\underline{a}}}^*:=\{\bfs F_s\in\mathcal{F}_{d}^s: Z(\bfs
F_s(\bfs a_i,-))=0\ (1\le i\le h-1),\ Z(\bfs F_s(\bfs a_h,-))>0\}.
\end{equation}
For this purpose, according to \eqref{eq: S_h in terms of S_k}, we
have
\begin{align*}
|S_{h,\bfs{\underline{a}}}^*|&=|S_{1,\bfs
a_h}^*|+\sum_{k=1}^{h-1}(-1)^k\sum_{1\le i_1<\cdots<i_k<
h}\big|S_{k+1,(\bfs a_{i_1},\ldots,\bfs a_{i_k},\bfs a_h)}\big|.
\end{align*}
We have the following result.
\begin{theorem}\label{th: estimate probability S_{h,a}^*}
Let $s<d$ and denote
$$s_d:=\sum_{i=1}^d(-1)^{i-1}\mbox{${q^s\choose i}$}q^{-si},\
s_{d+1}^+:=\sum_{i=1}^{d+1}\mbox{${q^s\choose i}$}q^{-si}\textrm{
and }\ t_{d+1}:=\mbox{${q^s\choose d+1}$}q^{-s(d+1)}.$$ For $2\le
h\le r-s$, we have
\begin{align*}
\Bigg|\frac{|S_{h,\bfs{\underline{a}}}^*|}{|\mathcal{F}_d^s|}-
s_d(1-s_d)^{h-1}\Bigg|&\le
t_{d+1}\big((1+s_{d+1}^+)^{h-1}+\mbox{$\frac{1}{2}$} \big)\ \textrm{
for }d\textrm{ odd},\\[1ex]
%
%
\Bigg|\frac{|S_{h,\bfs{\underline{a}}}^*|}{|\mathcal{F}_d^s|}-
s_{d+1}(1-s_{d+1})^{h-1}\Bigg|&\le
t_{d+1}\big((1+s_{d+1}^+)^{h-1}+\mbox{$\frac{1}{2}$} \big)\ \textrm{
for }d\textrm{ even}.
\end{align*}
\end{theorem}
\begin{proof}
Assume that $d$ is odd. Observe that
$$s_d(1-s_d)^{h-1}=s_d+\sum_{k=1}^{h-1}\binom{h-1}{k}(-1)^ks_d^{k+1}
=s_d+\sum_{k=1}^{h-1}(-1)^k\sum_{1\le i_1<\cdots<i_k< h}s_d^{k+1}.
$$
According to Theorem \ref{th: probability 1 specialization} and
Proposition \ref{prop: bound for S_k},
$$\Bigg|\frac{\big|S_{k+1,(\bfs a_{i_1},\ldots,\bfs a_{i_k},\bfs
a_h)}\big|}{|\mathcal{F}_d^s|}-
s_d^{k+1}\Bigg|\le\mbox{$\frac{t_{d+1}}{2}$}\big((s_{d+1}^+)^k+(2k+1)s_d^k\big)$$
for any $1\le i_1<\cdots<i_k< h$. It follows that
\begin{align*}
\Bigg|\frac{|S_{h,\bfs{\underline{a}}}^*|}{|\mathcal{F}_d^s|}-s_d(1-s_d)^{h-1}\Bigg|&\le
\mbox{$\frac{t_{d+1}}{2}$}\sum_{k=0}^{h-1}\mbox{${h-1\choose
k}$}\big((s_{d+1}^+)^k+(2k+1)s_d^k\big).
\end{align*}

Now we bound the sums in the right-hand side of the previous
inequality. First, we have
$$\sum_{k=0}^{h-1}\mbox{${h-1\choose
k}$}(s_{d+1}^+)^k=(1+s_{d+1}^+)^{h-1}.$$
On the other hand,
\begin{align*}
\sum_{k=0}^{h-1}\mbox{${h-1\choose k}$}(2k+1)s_d^k&=
2s_d\sum_{k=1}^{h-1}\mbox{${h-1\choose k}$}ks_d^{k-1}+
\sum_{k=0}^{h-1}\mbox{${h-1\choose k}$}s_d^k&
\\
&=  2s_d(h-1)(1+s_d)^{h-2}+(1+s_d)^{h-1}\\& =
\big((2h-1)s_d+1\big)(1+s_d)^{h-2}.
\end{align*}
We conclude that
$$\Bigg|\frac{|S_{h,\bfs{\underline{a}}}^*|}{|\mathcal{F}_d^s|}-
s_d(1-s_d)^{h-1}\Bigg|\le\mbox{$\frac{t_{d+1}}{2}$}
\big((1+s_{d+1}^+)^{h-1}+\big((2h-1)s_d+1\big)(1+s_d)^{h-2}\big).$$

We claim that
$$-1+\big((2h-1)s_d+1\big)(1+s_d)^{h-2}\le (1+s_{d+1}^+)^{h-1}.$$
Indeed, for $d=3$ and for $h=2,3,4,5$ this can be checked by a
direct computation. For $d\ge 5$ and $h\ge 4$, the claim is a
consequence of the inequality $(2h-1)(1+s_d)^{h-1}\le
(1+s_{d+1}^+)^{h-1}$, which is equivalent to
$$2h-1\le\Big(\mbox{$\frac{1+s_{d+1}^+}{1+s_d}$}\Big)^{h-1}
=\Big(\mbox{$1+\frac{s_{d+1}^+-s_d}{1+s_d}$}\Big)^{h-1}.$$
Now, as $d\ge 5$, we have
$$\mbox{$1+\frac{s_{d+1}^+-s_d}{1+s_d}$}
\ge 1+\mbox{$\frac{2s_{d,\textrm{even}}+t_{d+1}}{1+1/e}$} \ge
1+\mbox{$\frac{2(t_2+t_4)+t_6}{1+1/e}$}\ge
1+\mbox{$\frac{2(\frac{5}{12}+\frac{5}{432})+\frac{1}{6^6}}{1+1/e}$}
\ge 1+\mbox{$\frac{\frac{43}{50}}{1+1/e}$},$$
and the inequality
$\big(1+\mbox{$\frac{\frac{43}{50}}{1+1/e}$}\big)^{h-1}\ge 2h-1$
holds for $h\ge 6$, which proves the claim.

By the claim we deduce that
$$\Bigg|\frac{|S_{h,\bfs{\underline{a}}}^*|}{|\mathcal{F}_d^s|}-
s_d(1-s_d)^{h-1}\Bigg|\le t_{d+1}
\big((1+s_{d+1}^+)^{h-1}+\mbox{$\frac{1}{2}$}\big),$$
which implies the bound in the theorem for $d$ odd. The bound for
$d$ even follows by a similar argument.
\end{proof}
%
%
\subsubsection{Probability that $h$ specializations are required}
Now  we are able to complete the analysis of the probability that
$h$ specializations are required. For this purpose, we fix $h$ with
$2\le h\le r-s+1$, denote
$$\FF_h:=\{(\bfs a_1\klk \bfs
a_h)\in\fq^{r-s}\times\cdots\times\fq^{r-s}:\bfs a_i\not=\bfs
a_j\textrm{ for }i\not=j\}, \quad N_h:=|\FF_h|,$$
and introduce the random variable $C_h:=\FF_h\times
\mathcal{F}^s_d\rightarrow \{1,\ldots,h, \infty\}$ defined as
follows:
\begin{equation}\label{eq: definition C_h}
C_h((\bfs a_1,\ldots,\bfs a_h),\bfs F_s):=\left\{
\begin{array}{rl}
\min &\!\!\!\!\{j:Z(\bfs F_s(\bfs a_j,-))>0\}\mbox{ if }\exists j\mbox{ with }Z(\bfs F_s(\bfs a_j,-))>0; \\
\infty,\!\!\!\! & \hbox{ otherwise.}
\end{array}
\right.
\end{equation}
We consider the set $\FF_h\times \mathcal{F}^s_d$ endowed with the
uniform probability $P_h:=P_{h,r,s,d}$ and study the probability of
the set $\{C_h=h\}$, that is, we aim to estimate the probability
$P_h(S_h^*)$, where
$$S_h^*:=\{((\bfs a_1,\ldots,\bfs a_h),\bfs F_s) \in \FF_h \times \mathcal{F}^s_d:
Z(\bfs F_s(\bfs a_i,-))=0\ (1\le i<h),\ Z(\bfs F_s(\bfs
a_h,-))>0\}.$$
We have the following result.
\begin{theorem}\label{th: probability of h specializations}
With assumptions as in Theorem \ref{th: estimate probability
S_{h,a}^*}, suppose further that $q^s>d$ and $1<h\le r-s+1$. We have
\begin{align*}
\big|P_h[C_h=h]-s_d(1-s_d)^{h-1}\big| &\le
t_{d+1}\big((1+s_{d+1}^+)^{h-1} +\mbox{$\frac{1}{2}$}\big)
+\mbox{$\frac{2}{q}$}\ \textrm{ for }d\textrm{ odd},\\[1ex]
\big|P_h[C_h=h]-s_{d+1}(1-s_{d+1})^{h-1}\big| &\le
t_{d+1}\big((1+s_{d+1}^+)^{h-1} +\mbox{$\frac{1}{2}$}\big)
+\mbox{$\frac{2}{q}$}\ \textrm{ for }d\textrm{ even}.
\end{align*}
\end{theorem}
\begin{proof}
Assume that $d$ is odd. Denote by $\mathcal{G}_d\subset\FF_h$ the
set of $\bfs{\underline{a}}$ satisfying the hypothesis of
Proposition \ref{prop: ranks generalized Vandermonde}. We have
\begin{align*}
P_h[C_h=h]=\frac{|S_h^*|}{N_h|\mathcal{F}^s_d|}=
\frac{1}{N_h|\mathcal{F}^s_d|}
\Bigg(\sum_{\bfs{\underline{a}}\in\FF_h\setminus\mathcal{G}_d}|S_{h,\bfs{\underline{a}}}^*|+
\sum_{\bfs{\underline{a}}\in\mathcal{G}_d}|S_{h,\bfs{\underline{a}}}^*|\Bigg),
\end{align*}
where $S_{h,\bfs{\underline{a}}}^*$ is defined as in \eqref{eq: def
S_h^*}.

From Remark \ref{rem: cardinality elements bad vandermonde rank} we
see that
\begin{align*}
\frac{N_h-|\mathcal{G}_d|}{N_h}&=
1-\frac{q^{h(r-s)}}{N_h}\bigg(1-\frac{1}{q}\bigg)\cdots
\bigg(1-\frac{1}{q^{h-1}}\bigg)\le
\frac{1}{q}+\cdots+\frac{1}{q^{h-1}}\le \frac{2}{q},\\
\frac{|\mathcal{G}_d|}{N_h}&=
\frac{q^{h(r-s)}}{N_h}\bigg(1-\frac{1}{q}\bigg)\cdots
\bigg(1-\frac{1}{q^{h-1}}\bigg)\ge
1-\frac{1}{q}-\cdots-\frac{1}{q^{h-1}}\ge 1-\frac{2}{q}.
\end{align*}
%
It follows that
$$\frac{1}{N_h|\mathcal{F}^s_d|}
\sum_{\bfs{\underline{a}}\in\FF_h\setminus\mathcal{G}_d}|S_{h,\bfs{\underline{a}}}^*|
= \frac{1}{N_h}
\sum_{\bfs{\underline{a}}\in\FF_h\setminus\mathcal{G}_d}
\frac{|S_{h,\bfs{\underline{a}}}^*|}{|\mathcal{F}^s_d|} \le
\frac{N_h-|\mathcal{G}_d|}{N_h} \le\frac{2}{q}.$$

On the other hand, by Theorem \ref{th: estimate probability
S_{h,a}^*} we obtain
\begin{align*}
\frac{1}{N_h}& \sum_{\bfs{\underline{a}}\in\mathcal{G}_d}
\frac{|S_{h,\bfs{\underline{a}}}^*|}{|\mathcal{F}^s_d|}\le
s_d(1-s_d)^{h-1}+t_{d+1}\big((1+s_{d+1}^+)^{h-1}+\mbox{$\frac{1}{2}$}\big)
\\
\frac{1}{N_h}& \sum_{\bfs{\underline{a}}\in\mathcal{G}_d}
\frac{|S_{h,\bfs{\underline{a}}}^*|}{|\mathcal{F}^s_d|}\ge
\Big(s_d(1-s_d)^{h-1}-t_{d+1}\big((1+s_{d+1}^+)^{h-1}+\mbox{$\frac{1}{2}$}
\big)\Big)\big(1-\mbox{$\frac{2}{q}$}\big).
\end{align*}
Putting these estimates together, and taking into account that
$s_d(1-s_d)^{h-1}\le \frac{1}{4}$ for $h>1$, we readily deduce the
statement of the theorem. The case $d$ even follows similarly.
\end{proof}

Finally, we express the estimates of Theorem \ref{th: probability of
h specializations} asymptotically for large $q$.
\begin{corollary}
\label{coro: probability of h specializations - asymptotic} With
notations and hypotheses as in Theorem \ref{th: probability of h
specializations}, assume further that $q^s
> 6$. We have
\begin{align*}
\big|P_h[C_h=h]-\mu_d(1-\mu_d)^{h-1}\big| &\le
\mbox{$\frac{1}{(d+1)!}$}\big(e^{h-1} +\mbox{$\frac{1}{2}$}\big)
+\mbox{$\frac{2}{q}$}+ \mbox{$\frac{5}{q^s}$}(2-\mu_d)^{h-1}
\ \textrm{ for }d\textrm{ odd},\\[1ex]
\big|P_h[C_h=h]-\mu_{d+1}(1-\mu_{d+1})^{h-1}\big| &\le
\mbox{$\frac{1}{(d+1)!}$}\big(e^{h-1} +\mbox{$\frac{1}{2}$}\big)
+\mbox{$\frac{2}{q}$} + \mbox{$\frac{5}{q^s}$}(2-\mu_d)^{h-1}\
\textrm{ for }d\textrm{ even}.
\end{align*}
\end{corollary}
\begin{proof}
Since $\binom{q^s}{j}q^{-sj}\le \mbox{$\frac{1}{j!}$}$ for every
$j\geq 0$, we have $s^+_{d+1}\le e-1$ and $t_{d+1}\le
\mbox{$\frac{1}{(d+1)!}$}$. It follows that
\begin{equation}\label{eq:1}
t_{d+1}\big((1+s_{d+1}^+)^{h-1} +\mbox{$\frac{1}{2}$}\big)
+\mbox{$\frac{2}{q}$}\le \mbox{$\frac{1}{(d+1)!}$}\big(e^{h-1}
+\mbox{$\frac{1}{2}$}\big) +\mbox{$\frac{2}{q}$}.
\end{equation}
Now assume that $d$ is odd. According to \eqref{eq: s_m asymptotic},
$$
|s_d-\mu_d|\le \frac{1}{4q^s}+\frac{d}{q^{2s}}<\frac{2}{q^s}.
$$
As a consequence,
\begin{align*}
s_d(1-s_d)^{h-1}&<
\big(\mu_d+\mbox{$\frac{2}{q^s}$}\big)\big(1-\mu_d+\mbox{$\frac{2}{q^s}$}\big)^{h-1}
\\
&=\big(\mu_d+\mbox{$\frac{2}{q^s}$}\big)\bigg(
(1-\mu_d)^{h-1}+\sum_{j=0}^{h-2}
\mbox{$\binom{h-1}{j}$}(1-\mu_d)^j\big(\mbox{$\frac{2}{q^s}$}\big)^{h-1-j}\bigg).
\end{align*}
As  $\mbox{$\frac{2}{q^s}$}<1$, we have
$$
\sum_{j=0}^{h-2}\mbox{$\binom{h-1}{j}$}(1-\mu_d)^j\big(\mbox{$\frac{2}{q^s}$}\big)^{h-1-j}
\le
\mbox{$\frac{2}{q^s}$}\sum_{j=0}^{h-2}\mbox{$\binom{h-1}{j}$}(1-\mu_d)^j
\le \mbox{$\frac{2}{q^s}$}(2-\mu_d)^{h-1}.
$$
Thus,
\begin{align*}
s_d(1-s_d)^{h-1}  \le \big(\mu_d+\mbox{$\frac{2}{q^s}$}\big)\big((1-\mu_d)^{h-1}+\mbox{$\frac{2}{q^s}$}(2-\mu_d)^{h-1}\big)
 \le \mu_d(1-\mu_d)^{h-1} + \mbox{$\frac{5(2-\mu_d)^{h-1}}{q^s}$}.
\end{align*}
Combining \eqref{eq:1} with this inequality and Theorem \ref{th:
probability of h specializations} we obtain
$$
P_h[C_h=h]\le \mu_d(1-\mu_d)^{h-1} +
\mbox{$\frac{5(2-\mu_d)^{h-1}}{q^s}$} +
\mbox{$\frac{1}{(d+1)!}$}\big(e^{h-1} +\mbox{$\frac{1}{2}$}\big)
+\mbox{$\frac{2}{q}$}.
$$

To prove the lower bound for $P_h[C_h=h]$ we observe that, as $d>1$
is odd, then $d\geq3$. Since
$\mu_d=\sum_{j=1}^{\infty}\frac{(-1)^{j-1}}{j!}$, it is easy to see
that $ \frac{1}{2}< \mu_d < \frac{2}{3}$, and the hypothesis $q^s >
6$ implies $\mu_d-\frac{2}{q^s}>0$ and $1-\mu_d-\frac{2}{q^s}>0$. We
conclude that
$$
s_d(1-s_d)^{h-1}\geq\bigg(\mu_d-\frac{2}{q^s}\bigg)
\bigg(1-\mu_d-\frac{2}{q^s}\bigg)^{h-1}.
$$
Arguing as above we show that
$$
\bigg(\mu_d-\frac{2}{q^s}\bigg)\bigg(1-\mu_d-\frac{2}{q^s}\bigg)^{h-1}
\geq \mu_d(1-\mu_d)^{h-1}-\frac{4}{q^s}(2-\mu_d)^{h-1},
$$
which proves the lower bound. A similar analysis works for $d$ even.
\end{proof}
%
%
\subsubsection{Probability that more specializations are required}
\label{subsubsec: prob more specializ are required}
Theorems \ref{th: probability 1 specialization} and \ref{th:
probability of h specializations} provide estimates on the
probability that $h$ specializations are required for $1\le h\le
r-s+1$. These estimates show that such a probability decreases
exponentially with $h$. For more specializations this exponential
behavior might be altered, as further conditions imposed on the
elements of $\mathcal{F}_d^s$ might not be linearly independent. For
this reason, in this section we estimate the probability that more
than $r-s+1$ specializations are required.

Similarly to previously, for any $h\ge 3$ we denote
\begin{align*}
\FF_h:=\{(\bfs a_1\klk \bfs
a_h)\in\fq^{r-s}\times\cdots\times\fq^{r-s}:\bfs a_i\not=\bfs
a_j\textrm{ for }i\not=j\}, \quad N_h:=|\FF_h|,
\end{align*}
and consider the random variable $C_h:=C_{h,r,d}: \FF_h\times
\mathcal{F}_{d}^s\to\{1\klk h,\infty\}$ defined for $\underline{\bfs
a}:=(\bfs a_1\klk\bfs a_h)\in\FF_h$ and $\bfs F_s\in\mathcal{F}_d^s$
in the following way:
$$C_h(\underline{\bfs
a},\bfs F_s):=\left\{\begin{array}{l} \min\{j: Z(\bfs F_s(\bfs
a_j,-))>0\}\ \textrm{ if }\exists j\textrm{ with }Z(\bfs F_s(\bfs
a_j,-))>0,\\[0.25ex]
\qquad\infty \qquad\textrm{otherwise}.\\
\end{array}\right.$$
We consider the set $\FF_h\times\mathcal{F}_{d}^s$ as before endowed
with the uniform probability $P_h:=P_{h,r,d}$ and analyze the
probability $P_h[C_h=h]$. Up to now, all the probability estimates
concerned the probability spaces $\FF_h\times \mathcal{F}_d^s$
separately. To link the probabilities $P_h$ for $1\le h\le q^{r-s}$,
we have the following result.
\begin{lemma}\label{lemma: consistency conditions}
Let $h>1$ and let $\pi_h:\FF_h\times\mathcal{F}_{d}^s
\to\FF_{h-1}\times\mathcal{F}_{d}^s$ be the mapping induced by the
projection $\FF_h\to\FF_{h-1}$ on the first $h-1$ coordinates. If
$\mathcal{S}\subset\FF_{h-1}\times\mathcal{F}_{d}^s$, then
$P_h[\pi_h^{-1}(\mathcal{S})]=P_{h-1}[\mathcal{S}].$
\end{lemma}
\begin{proof}
Note that
\begin{align*}
\pi_h^{-1}(\mathcal{S})&= \bigcup_{\bfs F_s\in\mathcal{F}_d^s}
\{(\bfs a_1\klk \bfs a_h)\in\FF_h:(\bfs
a_1\klk \bfs a_{h-1},\bfs F_s)\in\mathcal{S}\}\times\{\bfs F_s\}\\
&=\bigcup_{\bfs F_s\in\mathcal{F}_d^s}
\bigcup_{\stackrel{\scriptstyle(\bfs a_1\klk\bfs
a_{h-1})\in\FF_{h-1}:}{(\bfs a_1\klk\bfs a_{h-1},\bfs
F_s)\in\mathcal{S}}}\{(\bfs a_1\klk\bfs a_{h-1})\}\times
(\fq^{r-s}\setminus\{\bfs a_1\klk\bfs a_{h-1}\})\times \{\bfs F_s\}.
\end{align*}
It follows that
\begin{align*}
P_h[\pi_h^{-1}(\mathcal{S})]&= \frac{1}{N_h|\mathcal{F}_d^s|}
\sum_{\bfs F_s\in\mathcal{F}_d^s} \sum_{\underline{\bfs
a}\in\FF_{h-1}:(\underline{\bfs a},\bfs
F_s)\in\mathcal{S}}(q^{r-s}-h+1)\\&=
\frac{1}{N_{h-1}|\mathcal{F}_d^s|} \sum_{\bfs
F_s\in\mathcal{F}_d^s}\big|\{\underline{\bfs
a}\in\FF_{h-1}:(\underline{\bfs a},\bfs F_s)\in\mathcal{S}\}\big|=
P_{h-1}[\mathcal{S}].
\end{align*}
This proves the lemma.
\end{proof}

According to the Kolmogorov extension theorem (see, e.g.,
\cite[Chapter IV, Section 5, Extension Theorem]{Feller71}), the
conditions of ``consistency'' of Lemma \ref{lemma: consistency
conditions} imply that the probabilities $P_h$ ($1\le h\le q^{r-s}$)
can be put in a unified framework. More precisely, we define
$\FF:=\FF_{\!q^{r-s}}$ and $P:=P_{q^{r-s}}$. Then the probability
measure $P$ defined on $\FF$ allows us to interpret consistently all
the results of this paper. In the same vein, the variables $C_h$
($1\le h\le q^{r-s}$) can be naturally extended to a random variable
$C: \FF\times\mathcal{F}_d^s \to\N\cup\{\infty\}$. Consequently, in
what follows we shall drop the subscript $h$ from the notations
$P_h$ and $C_h$.

Now we are able to state and prove our result on the probability
that more than $r-s+1$ random choices are required.
\begin{corollary}\label{coro: prob C>h^*}
With assumptions as in Theorem \ref{th: estimate probability
S_{h,a}^*}, suppose further that $q^s>d$ and denote $h^*:=r-s+1$. We
have
\begin{align*}
\big| P[C>h^*]-(1-\mu_d)^{h^*}\big| &\le
\mbox{$\frac{e^{h^*}}{(d+1)!}$} +\mbox{$\frac{2h^*}{q}$}
+\mbox{$\frac{15(2-\mu_d)^{h^*}}{q^s}$}\textrm{ for }d\textrm{
odd},\\[1ex]
\big| P[C>h^*]-(1-\mu_{d+1})^{h^*}\big| &\le
\mbox{$\frac{e^{h^*}}{(d+1)!}$} +\mbox{$\frac{2h^*}{q}$}
+\mbox{$\frac{15(2-\mu_d)^{h^*}}{q^s}$}\textrm{ for }d\textrm{
even}.
\end{align*}
\end{corollary}
\begin{proof}
Suppose that $d$ is odd. Observe that
$P[C>h^*]=1-\sum_{h=1}^{h^*}P[C=h]$ and
%
$$(1-\mu_d)^{h^*}=1-\sum_{h=1}^{h^*}\mu_d(1-\mu_d)^{h-1}.$$
%
According to Corollaries \ref{coro: probability 1 specialization -
asymptotic} and \ref{coro: probability of h specializations -
asymptotic},
\begin{align*}
\big|P[C>h^*]-(1-\mu_d)^{h^*}\big|&\le \sum_{h=1}^{h^*}\big|P[C=h]-\mu_d(1-\mu_d)^{h-1}|\\
&\le \mbox{$\frac{1}{(d+1)!}$}\bigg(\mbox{$\frac{h^*-1}{2}$}
+\sum_{h=1}^{h^*}e^{h-1}\bigg)+\mbox{$\frac{2(h^*-1)}{q}$}+\mbox{$\frac{5}{q^s}$} \sum_{h=1}^{h^*}(2-\mu_d)^{h-1}\\
&\le \mbox{$\frac{e^{h^*}}{(d+1)!}$}+\mbox{$\frac{2h^*}{q}$}
+\mbox{$\frac{15(2-\mu_d)^{h^*}}{q^s}$}.
\end{align*}
This proves the corollary.
\end{proof}
%
%
\section{Average-case complexity and probability of success}
\label{section: analysis of SVS algorithm}
Now we are finally able to provide estimates on the average-case
complexity and the probability of success of Algorithm \ref{algo:
basic scheme}. For $1<s<r$ and $d \geq 2$, we recall that this
algorithms takes as input an element of $\mathcal{F}^s_d$, namely an
$s$-tuple $\bfs F_s:=(F_1, \dots, F_s)$ of elements of
$\mathcal{F}_{d}:=\{F\in \fq[X_1, \dots, X_r]:\deg F \le d\}$, and
outputs an $\fq$-rational solution $\bfs x\in\fq^r$ of the system
$\bfs F_s=\bfs 0$, or ``{\tt Failure}''.

For this purpose, 
Algorithm \ref{algo: basic scheme} successively generates a sequence
$\underline{\bfs a}:=(\bfs a_1,\bfs a_2,\ldots,\bfs a_{h^*})\in
\FF_{h^*}$, where $h^*:=r-s+1$, and searches for $\fq$--rational
zeros of $\bfs F_s$ in the vertical strips $\{\bfs a_i\}\times\fq^s$
for $1\le i\le h^*$, until a zero of $\bfs F_s(\bfs a_i,-)$ is found
or all the vertical strips are exhausted. As discussed in Section
\ref{section: intro}, given $\underline{\bfs a}\in \FF_{_{h^*}}$,
the whole procedure requires at most $C_{\underline{\bfs a}}(\bfs
F_s)\cdot \tau(d,s,q)$ arithmetic operations in $\fq$, where
$\tau(d,s,q)$ is the maximum number of arithmetic operations in
$\fq$ necessary to perform a search in an arbitrary vertical strip,
provided that $i^*:=C(\underline{\bfs a},\bfs F_s)\le h^*$, $\bfs
F_s(\bfs a_{i^*},-)$ satisfies hypothesis ({\sf H}) and the choice
of parameters $\omega$ for the probabilistic algorithm performing
the zero-dimensional search on the specialized system $\bfs F_s(\bfs
a_{i^*},-)=\bfs 0$ is accurate. The analysis of the probability that
all these conditions hold is essential to determine the average-case
complexity and probability of success of Algorithm \ref{algo: basic
scheme}, as we shall see in the next sections.
%
%
\subsection{Average-case complexity}
To perform a search in a vertical strip $\{\bfs a\}\times \fq^s$ we
shall use Gr\"obner-basis technology or a Kronecker-like algorithm.
In the first case, we shall use a deterministic algorithm, combined
with a probabilistic routine for computing $\fq$-rational roots of
univariate polynomials. In the second case, we shall use a
probabilistic algorithm. Therefore, in both cases we rely on $r_d$
random choices of elements of $\fq$ for certain $r_d\in\N$. We
denote by $\Omega_d:=\fq^{r_d}$ the set of all such random choices
and consider $\Omega_d$ endowed with the uniform probability,
$\FF_{h^*}\times\mathcal{F}_d^s$ with the (uniform) probability $P$
of Section \ref{subsubsec: prob more specializ are required}, and
$\FF_{h^*}\times\mathcal{F}_d^s\times \Omega_d$ with the product
probability. Therefore, we may represent the cost of Algorithm
\ref{algo: basic scheme} either using Gr\"obner bases or
Kronecker-like algorithms by the random variable
\begin{align*}
X:= X_{s,r,d} : \FF_{h^*}\times\mathcal{F}_d^s\times
\Omega_d\rightarrow \N
\end{align*}
which counts the number $X(\underline{\bfs a},\bfs F_s,\omega)$ of
arithmetic operations performed on input $\bfs
F_s\in\mathcal{F}_d^s$, with the choice of vertical strips defined
by $\underline{\bfs a}$ and the choice $\omega$ for the parameters
of the routine for zero-dimensional solving. We shall use versions
of these routines having a probability of failure of order
$\mathcal{O}(q^{-1})$. This can be achieved by a modification of the
algorithm for computing $\fq$-rational roots of, e.g.,
\cite[Algorithm 14.15]{GaGe99}, and the Kronecker-like algorithm in
\cite{CaMa06a}, which consists of performing the random choices on a
set $\Omega_d$ suitably augmented, with an increase of the number of
arithmetic operations in $\fq$ of the corresponding algorithms by a
factor $\mathcal{O}(\log q)$.

We also remark that the search for $\fq$-rational solutions in an
arbitrary vertical strip $\{\bfs a\}\times\fq^s$ will be truncated
when $\tau(s,d,q)$ arithmetic operations in $\fq$ are performed, so
that at most $\tau(s,d,q)$ arithmetic operations in $\fq$ are
performed during all the searches of Algorithm \ref{algo: basic
scheme}. For Gr\"obner-basis technology, this can be done by
establishing bounds on the degrees of the polynomials arising in the
underlying Macaulay matrix, while for Kronecker algorithms it can be
done by controlling the degrees in the intermediate varieties
arising during the execution of the algorithm.

Having made the algorithmic model precise, we proceed to bound from
above the asymptotic behavior of the expected value $E[X]$ of $X$,
namely
$$E[X]:=\frac{1}{|\FF_{h^*}||\mathcal{F}_d^s||\Omega_d|}
\sum_{(\underline{\bfs a},\bfs F_s,\omega)}X(\underline{\bfs a},\bfs
F_s,\omega).$$
\begin{theorem}\label{th: average-case compl}
Let $h^*:={r-s+1}$. For $q>2d^s(d+1)^s$ and $d>s$, the average--case
complexity of Algorithm \ref{algo: basic scheme} is bounded in the
following way:
\begin{align*}
E[X] &\leq {\tau(d,s,q)} \Big(\mu_d^{-1}+h^*(1-\mu_d)^{h^*}\!\!+
\mbox{$\frac{3{h^*}e^{h^*}}{(d+1)!}$}+\mathcal{O}\big(\mbox{$\frac{h^*
d^s(d+1)^s}{q}$}+\mbox{$\frac{h^*
(2-\mu_d)^{h^*}\!\!}{q^s}$}\big)\Big)\textrm{ for }d\textrm{ odd},\\[1ex]
E[X] &\leq {\tau(d,s,q)}
\Big(\mu_{d+1}^{-1}+h^*(1-\!\mu_{d+1})^{h^*}\!\!+
\mbox{$\frac{3{h^*}e^{h^*}}{(d+1)!}$}+\mathcal{O}\big(\mbox{$\frac{h^*
d^s(d+1)^s}{q}$}+\mbox{$\frac{h^*
(2-\mu_d)^{h^*}\!\!}{q^s}$}\big)\Big)\textrm{ for }d\textrm{ even},
\end{align*}
where $\tau(d,s,q)$ is the cost of the search in a vertical strip
and the constant underlying the $\mathcal{O}$-notation is
independent of $r$, $s$, $d$ and $q$.
\end{theorem}
\begin{proof}
We first note that to each $(\underline{\bfs a},\bfs F_s)\in
\FF_{h^*}\times \mathcal{F}_d^s$ and $i^*:=C_{h^*}(\underline{\bfs
a},\bfs F_s)\le h^*$, such that $\bfs F_s(\bfs a_{i^*}, -)$
satisfies condition $({\sf H})$, there corresponds a subset
$\Omega_d(\underline{\bfs a},\bfs F_s)$ of ``bad'' choices of
parameters $\omega$ for the zero--dimensional solver under
consideration, that is, choices for which the solver fails to find
an $\fq$--rational solution of the system $\bfs  F_s(\bfs a_{i^*},
-)=\bfs 0$. According to our previous remarks, the probability of
failure is $|\Omega_d(\underline{\bfs a},\bfs
F_s)|/{|\Omega_d|}=\mathcal{O}(q^{-1})$.

We decompose $\FF_{h^*}\times \mathcal{F}_d^s\times \Omega_d$ as the
disjoint union
$$\FF_{h^*}\times \mathcal{F}_d^s\times \Omega_d={\sf C} \cup {\sf D} \cup {\sf E} \cup {\sf F},$$
where
\begin{align*}
 {\sf C}&:=\{(\underline{\bfs a},\bfs F_s, \omega): C_{h^*}(\underline{\bfs a},\bfs F_s)=\infty \},\\
 {\sf D}&:=\{(\underline{\bfs a},\bfs F_s, \omega): i^*:=C_{h^*}(\underline{\bfs a},\bfs F_s)\le h^*
 \textrm{ and } \bfs  F_s(\bfs a_{i^*}, -) \textrm{ does not satisfy } ({\sf H}) \}, \\
 {\sf E}&:=\{(\underline{\bfs a},\bfs F_s, \omega): i^*:=C_{h^*}(\underline{\bfs a},\bfs F_s)\le h^*,
 \bfs  F_s(\bfs a_{i^*}, -) \textrm{ satisfies } ({\sf H}) \textrm{ and } \omega \in
 \Omega_d(\underline{\bfs a},\bfs F_s) \}, \\
 {\sf F}&:=\{(\underline{\bfs a},\bfs F_s, \omega): i^*:=C_{h^*}(\underline{\bfs a},\bfs F_s)\le h^*,
 \bfs  F_s(\bfs a_{i^*}, -) \textrm{ satisfies } ({\sf H}) \textrm{ and } \omega \not \in \Omega_d(\underline{\bfs a},\bfs F_s) \}.
\end{align*}
For each $(\underline{\bfs a},\bfs F_s, \omega)\in {\sf F}$, when
the zero--dimensional solver under consideration is applied to the
system $\bfs F_s(\bfs a_{i^*}, -)=\bfs 0$, it succeeds. Then
Algorithm \ref{algo: basic scheme} stops at the $i^*$th vertical
strip and outputs an $\fq$-rational solution of the system $\bfs
F_s(\bfs a_{i^*}, -)=\bfs 0$, and we have
$$X(\underline{\bfs a},\bfs F_s,\omega)\le
\tau(s,d,q)\,C_{h^*}(\underline{\bfs a},\bfs F_s).$$
On the other hand, for any $(\underline{\bfs a},\bfs F_s, \omega)\in
{\sf C} \cup {\sf D} \cup {\sf E}$, Algorithm \ref{algo: basic
scheme} does not stop searching for $\fq$--rational solutions at the
$i^*$th vertical strip, and the search continues until an
$\fq$--rational solution of a system $\bfs F_s(\bfs a_{i}, -)=\bfs
0$ with $i^*<i\le h^*$ is obtained, or the remaining vertical strips
are searched. In this case we have $X(\underline{\bfs a},\bfs
F_s,\omega)\le \tau(s,d,q)h^*$. We conclude that
\begin{equation}\label{eq: S_A_upper_bound}
E[X] \le \frac{\tau(s,d,q)}{|\FF_{h^*}||\mathcal{F}_d^s||\Omega_d|}
\Bigg(\sum_{(\underline{\bfs a},\bfs F_s, \omega)\in {\sf F}}C_{h^*}
(\underline{\bfs a},\bfs F_s) + h^*\big(|{\sf C}| + |{\sf D}|+ |{\sf
E}|\big)\Bigg)
\end{equation}

Now we study the sum in the right--hand side of \eqref{eq:
S_A_upper_bound}. We have
\begin{align*}
\frac{\tau(s,d,q)}{|\FF_{h^*}||\mathcal{F}_d^s||\Omega_d|}\sum_{(\underline{\bfs
a},\bfs F_s, \omega)\in{\sf F}} C_{h^*}(\underline{\bfs a},\bfs F_s)
& \le
\frac{\tau(s,d,q)}{|\FF_{h^*}||\mathcal{F}_d^s|}\sum_{(\underline{\bfs
a},\bfs F_s):
\,  C_{h^*}(\underline{\bfs a},\bfs F_s)\le h^*}C_{h^*}(\underline{\bfs a},\bfs F_s)\\
& = \frac{\tau(s,d,q)}{|\mathcal{F}_d^s|} \sum_{\bfs F_s\in
\mathcal{F}_d^s} \sum_{h=1}^{h^*}h\frac{|\{\underline{\bfs a}\in
\FF_{h^*}: C_{h^*}(\underline{\bfs a},\bfs F_s)=h\}|}{|\FF_{h^*}|}.
\end{align*}
From the conditions of consistency of Lemma \ref{lemma: consistency
conditions}, it follows that
\begin{align*}
\frac{1}{|\FF_{h^*}||\mathcal{F}_d^s|}\sum_{(\underline{\bfs a},\bfs
F_s): \, C_{h^*}(\underline{\bfs a},\bfs F_s)\le
h^*}C_{h^*}(\underline{\bfs a},\bfs F_s) &=
\frac{1}{|\mathcal{F}_d^s|} \sum_{\bfs F_s\in \mathcal{F}_d^s}
\sum_{h=1}^{{h^*}}h\frac{|\{\underline{\bfs a}\in \FF_{h^*}:
C_{h*}(\underline{\bfs a},\bfs F_s)=h\}|}{|\FF_{h^*}|}
\\&=\sum_{h=1}^{h^*}\frac{h}{|\mathcal{F}_d^s|}\sum_{\bfs
F_s\in \mathcal{F}_d^s}\frac{|\{\underline{\bfs a}\in \FF_h:
C_h(\underline{\bfs a}, \bfs F_s)=h\}|}{|\FF_h|} \\  &=
\sum_{h=1}^{h^*}h\,P_{h}[C_h=h].
\end{align*}
%


%
Assuming that $d$ is odd, by Corollary \ref{coro: probability of h
specializations - asymptotic}  we have
\begin{align*}
\frac{1}{|\FF_{h^*}||\mathcal{F}_d^s|} &\sum_{(\underline{\bfs
a},\bfs F_s): \, C_{h^*}(\underline{\bfs a},\bfs F_s)\le
h^*}C_{h^*}(\underline{\bfs a},\bfs F_s)\\
&\qquad\quad \le \sum_{h=1}^{h^*}h\mu_d(1-\mu_d)^{h-1}
+\mbox{$\frac{2h^*e^{h^*}}{(d+1)!}$} + \mbox{$\frac{h^*(h^*+1)}{q}$}
+\mathcal{O}\big(\mbox{$\frac{h^*(2-\mu_d)^{h^*}}{q^{s}}$}\big)\\
& \qquad\quad\le \mu_d\sum_{h=1}^{h^*}h(1-\mu_d)^{h-1}
+\mbox{$\frac{2h^*e^{h^*}}{(d+1)!}$} + \mbox{$\frac{2{h^*}^2}{q}$}
+\mathcal{O}\big(\mbox{$\frac{h^*(2-\mu_d)^{h^*}}{q^{s}}$}\big).
\end{align*}
Taking into account that $\sum_{n\geq 1}n z^{n-1}=1/(1-z)^2$ for any
$|z|< 1$, we conclude that
\begin{equation}\label{eq: sum_over_F_upper_bound}
\frac{\tau(s,d,q)}{|\FF_{h^*}||\mathcal{F}_d^s||\Omega_d|}
\sum_{(\underline{\bfs a},\bfs F_s, \omega)\in
F}C_{h^*}(\underline{\bfs a},\bfs F_s) \le
\tau(s,d,q)\bigg(\frac{1}{\mu_d}+\mbox{$\frac{2h^*e^{h^*}}{(d+1)!}$}
+ \mbox{$\frac{2{h^*}^2}{q}$}
+\mathcal{O}\big(\mbox{$\frac{h^*(2-\mu_d)^{h^*}}{q^{s}}$}\big)\bigg).
\end{equation}

Next we consider the remaining terms in the right--hand side of
\eqref{eq: S_A_upper_bound}. We have
\begin{align*}
\frac{|{\sf C}|}{|\FF_{h^*}||\mathcal{F}_d^s||\Omega_d|}&
=\frac{|\{(\underline{\bfs a},\bfs F_s):  C_{h^*}(\underline{\bfs
a},\bfs F_s)=\infty\}|}{|\FF_{h^*}||\mathcal{F}_d^s|} =
P_{h^*}[C_{h^*}=\infty]=P[C > h^*].
\end{align*}

On the other hand,
\begin{align*}
\frac{|{\sf D}|}{|\FF_{h^*}||\mathcal{F}_d^s||\Omega_d|} &=
\frac{|\{(\underline{\bfs a},\bfs F_s): i^*:=C_{h^*}(\underline{\bfs
a},\bfs F_s)\le h^* \textrm{ and } \bfs  F_s(\bfs a_{i^*}, -)
\textrm{ does not satisfy } ({\sf H}) \}|
}{|\FF_{h^*}||\mathcal{F}_d^s|}\\
& =\sum_{j=1}^{h^*}\frac{|\{(\underline{\bfs a},\bfs F_s):
C_{h^*}(\underline{\bfs a},\bfs F_s)=j\textrm{ and } \bfs  F_s(\bfs
a_{j}, -) \textrm{ does not satisfy } ({\sf H})\}|}
{|\FF_{h^*}||\mathcal{F}_d^s|}\\
& \le \sum_{j=1}^{h^*}\frac{|\pi_j^{-1}(\{C_{\sf H}=\infty\})|
}{|\FF_{h^*}||\mathcal{F}_d^s|},
\end{align*}
where $\pi_j: \FF_{h^*}\times \mathcal{F}_d^s \rightarrow
\fq^{r-s}\times \mathcal{F}_d^s$ is the map $\pi_j(\underline{\bfs
a},\bfs F_s):= (\bfs a_j, \bfs F_s)$. Since
$$|\pi_j^{-1}(\bfs a, \bfs F_s)|=(q^{r-s}-1) \cdots (q^{r-s}-h^*+1)$$
for every $(\bfs a, \bfs F_s)\in \fq^{r-s}\times \mathcal{F}_d^s$,
it follows that
$$|\pi_j^{-1}(\{C_{\sf H}=\infty\})|=(q^{r-s}-1) \cdots
(q^{r-s}-h^*+1)|\{C_{\sf H}=\infty\}|.$$
As a consequence,
\begin{align*}
\frac{|{\sf D}|}{|\FF_{h^*}||\mathcal{F}_d^s||\Omega_d|}&\le
\frac{h^*(q^{r-s}-1) \cdots (q^{r-s}-h^*+1)|\{C_{\sf H}=\infty\}|
}{|\FF_{h^*}||\mathcal{F}_d^s|}= h^*P_1[C_{\sf H}=\infty].
\end{align*}

Finally, to estimate $|{\sf E}|$,  note that for each
$(\underline{\bfs a},\bfs F_s)\in \FF_{h^*}\times \mathcal{F}_d^s$
with $i^*:=C_{h^*}(\underline{\bfs a},\bfs F_s)\le h^*$, such that
$\bfs  F_s(\bfs a_{i^*}, -)$ satisfies $({\sf H})$, we have
$\frac{|\Omega_d(\underline{\bfs a},\bfs
F_s)|}{|\Omega_d|}=\mathcal{O}(q^{-1})$. It follows that
\begin{align*}
\frac{|{\sf E}|}{|\FF_{h^*}||\mathcal{F}_d^s||\Omega_d|}&=
\frac{|\{(\underline{\bfs a},\bfs F_s): i^*:=
C_{h^*}(\underline{\bfs a},\bfs F_s)\le h^* \textrm{ and }
\bfs  F_s(\bfs a_{i^*}, -) \textrm{ satisfies } ({\sf H}) \}|}{|\FF_{h^*}||\mathcal{F}_d^s|}\,\mathcal{O}\big(q^{-1}\big)\\
&= \mathcal{O}\big(q^{-1}\big).
\end{align*}

Combining the estimates for ${\sf C}$, ${\sf D}$ and ${\sf E}$ with
Corollaries \ref{coro: probability 1 specializ and red reg sequence}
and \ref{coro: prob C>h^*} we obtain
\begin{align*}
\frac{\tau(s,d,q)h^*}{|\FF_{h^*}||\mathcal{F}_d^s||\Omega_d|}&\big(|{\sf
C}| + |{\sf D}|+ |{\sf E}|\big)\\&\le  \tau(s,d,q)h^*\big(P[C > h^*]
+P_1[C_{\sf H}=\infty]+\mathcal{O}(q^{-1})\big)\\
& \le
\tau(s,d,q)h^*\Big((1-\mu_d)^{h^*}+\mbox{$\frac{e^{h^*}}{(d+1)!}$} +
\mbox{$\frac{2h^*}{q}$}+ \mbox{$\frac{15(2-\mu_d)^{h^*}}{q^s}$} +
\mathcal{O}\big(\mbox{$\frac{d^s(d+1)^s}{q}$}\big)\Big)
\end{align*}
Taking into account \eqref{eq: sum_over_F_upper_bound} and this
inequality we readily deduce the theorem for $d$ odd. The case $d$
even follows with a similar argument.
\end{proof}


We briefly mention the average-case complexity of Algorithm
\ref{algo: basic scheme} performing the zero-dimensional searches
with Gr\"obner bases and Kronecker-like algorithms.
\begin{corollary}\label{coro: average-case compl with GB}
Let notations and assumptions be as in Theorem \ref{th: average-case
compl}. Denote by $E_{\tt GB}[X]$ and $E_{\tt K}[X]$ the
average--case complexities of Algorithm \ref{algo: basic scheme},
performing zero-dimensional searches with Gr\"obner bases and the
Kronecker algorithm respectively. We have
\begin{align*}
E_{\tt GB}[X] &= \mathcal{O}^\sim\Big(\big(\mbox{$\binom{d+r}{r}+
d{sd+1\choose s}^\omega$}+d^{3s}+d^s\log
q\big)\big(1+\mbox{$\frac{h^*e^{h^*}}{(d+1)!}$}+\mbox{$\frac{d^{2s}}{q}$}+\mbox{$\frac{
h^*(2-\mu_d)^{h^*}\!\!}{q^s}$}\big)\Big),\\[1ex]
E_{\tt K}[X] &= \mathcal{O}^\sim \Big(\big(\mbox{$\binom{d+r}{r}+
\binom{d+s}{s}d^{2s}$}+d^s\log q\big)\big(
1+\mbox{$\frac{h^*e^{h^*}}{(d+1)!}$}+\mbox{$\frac{d^{2s}}{q}$}+\mbox{$\frac{
h^*(2-\mu_d)^{h^*}\!\!}{q^s}$}\big)\Big),
\end{align*}
where the notation $\mathcal{O}^\sim$ ignores logarithmic factors
and $\omega$ is the exponent of the complexity of multiplication of
square matrices with coefficients in $\fq$.
\end{corollary}
\begin{proof}
Denote by $\tau_{\tt GB}(d,s,q)$ the number of arithmetic operations
in $\fq$ required to perform a zero-dimensional search using
Gr\"obner bases, assuming that hypothesis $({\sf H})$ holds. As
explained in the introduction, if $D:=\binom{d+r}{r}$, then
$\mathcal{O}(sD)$ arithmetic operations in $\fq$ are performed to
compute the partial specialization $\bfs F_s(\bfs a,-)$, where $\bfs
F_s\in\mathcal{F}_d^s$ is the input system and $\bfs a\in\fq^{r-s}$
is the vertical strip under consideration. 
Then solving the zero-dimensional system $\bfs F_s(\bfs a,-)=\bfs 0$
for the degree reverse lexicographic order requires
$\mathcal{O}\big(s^2d{sd+1\choose s}^\omega\big)$ arithmetic
operations in $\fq$ (see, e.g., \cite{BaFaSa15} or \cite[Chapter
26]{BoChGiLeLeSaSc17}). Then the {\tt FGLM} algorithm (see
\cite{FaGiLaMo93}) is applied to solve the system $\bfs F_s(\bfs
a,-)=\bfs 0$ for the lexicographical order, with
$\mathcal{O}(sd^{3s})$ arithmetic operations in $\fq$. Finally, we
apply a routine for computing the $\fq$-rational roots of the
resulting univariate polynomial with $\mathcal{O}^\sim(d^s\log q)$
operations in $\fq$ (see, e.g., \cite[Algorithm 14.15]{GaGe99}). We
conclude that
$$\tau_{\tt GB}(d,s,q)\in \mathcal{O}^\sim\big(s\mbox{$D+
s^2d{sd+1\choose s}^\omega$}+sd^{3s}+d^s\log q\big)=
\mathcal{O}^\sim\big(\mbox{$D+ d{sd+1\choose
s}^\omega$}+d^{3s}+d^s\log q\big).$$

On the other hand, if $\tau_{\tt K}(d,s,q)$ is the number of
arithmetic operations in $\fq$ required to perform a
zero-dimensional search using the Kronecker algorithm, assuming that
hypothesis $({\sf H})$ holds, then, according to, e.g.,
\cite{CaMa06a}, one has
$$\tau_{\tt K}(d,s,q)\in \mathcal{O}^\sim\big(\mbox{$sD+
\binom{d+s}{s}d^{2s}$}+d^s\log q\big)=
\mathcal{O}^\sim\big(\mbox{$D+ \binom{d+s}{s}d^{2s}$}+d^s\log
q\big).$$
This proves the estimates in the corollary.
\end{proof}

Taking into account the bound ${e+t\choose t}\le \frac{3}{2}e^t$,
which holds for $e\ge 2$, we may simplify the bounds in the
corollary in the following way:
\begin{align*}
E_{\tt GB}[X] &\in \mathcal{O}^\sim\Big(\big(d^r+(sd)^{
s\omega}+d^s\log q\big)\big(1+
\mbox{$\frac{h^*e^{h^*}}{(d+1)!}$}+\mbox{$\frac{d^{2s}}{q}$}+\mbox{$\frac{
h^*(2-\mu_d)^{h^*}\!\!}{q^s}$}\big)\Big),\\[1ex]
E_{\tt K}[X] &\in\mathcal{O}^\sim \Big(\big(d^r+ d^{3s}+d^s\log
q\big)\big(1+
\mbox{$\frac{h^*e^{h^*}}{(d+1)!}$}+\mbox{$\frac{d^{2s}}{q}$}+\mbox{$\frac{
h^*(2-\mu_d)^{h^*}\!\!}{q^s}$}\big)\Big),
\end{align*}
%
%
\subsection{Probability of success}
We end with an estimate on the probability of success of our
algorithm. For this purpose, we first analyze the probability that,
given $h$ with $1\le h\le h^*$, after exactly $h$ random choices
$\bfs a_1,\ldots,\bfs a_h$ in $\fq ^{r-s}$ we obtain a system $\bfs
F_s(\bfs a_h,-)=\bfs 0$ with $\fq$-rational solutions which
satisfies hypothesis $({\sf H})$. More precisely, we consider the
random variable $C_h:=\FF_h\times \mathcal{F}^s_d\rightarrow
\{1,\ldots,h, \infty\}$ defined as in \eqref{eq: definition C_h} and
analyze the probability of the set 
$$\{C_h=h\}\cap S_{\sf H}^h,\quad S_{\sf H}^h:=\{(\underline{\bfs a},\bfs F_s)
\in\FF_h\times\mathcal{F}_d^s:\bfs F_s(\bfs a_h,-)\mbox{ satisfies
hypothesis }({\sf H})\}.$$
We have the following result.
\begin{lemma}\label{lemma: prob C_h=h and S_h}
For $q >2d^s(d+1)^s$, $s<d$ and $1<h\le h^*$ we have
\begin{align*}
\big|P_h[\{C_h=h\}\cap S_{\sf H}^h]-\mu_d(1-\mu_d)^{h-1}\big| &\le
\mbox{$\frac{e^{h-1} +\frac{1}{2}}{(d+1)!}$}
+\mbox{$\frac{2d^s(d+1)^s+2}{q}$}+\mbox{$\frac{5
(2-\mu_d)^{h-1}\!\!}{q^s}$}\textrm{ for }d\textrm{ odd},\\[1ex]
\big|P_h[\{C_h=h\}\cap S_{\sf
H}^h]-\mu_{d+1}(1-\mu_{d+1})^{h-1}\big| &\le \mbox{$\frac{e^{h-1}
+\frac{1}{2}}{(d+1)!}$}
+\mbox{$\frac{2d^s(d+1)^s+2}{q}$}+\mbox{$\frac{5
(2-\mu_d)^{h-1}\!\!}{q^s}$}\textrm{ for }d\textrm{ even}.
\end{align*}
\end{lemma}
\begin{proof}
For $d$ odd, according to the Fr\'echet inequalities,
$$
\max\{P_h[C_h=h]+P_h[S_{\sf H}^h]-1, 0\} \leq P_h[\{C_h=h\}\cap
S_{\sf H}^h] \leq \min \{P_h[C_h=h], P_h[S_{\sf H}^h]\}.$$

It is easy to see that $P_h[S_{\sf H}^h]=P_1[S_{\sf H}]$. Therefore,
by Corollary \ref{coro: probability 1 specializ and red reg
sequence} it follows that 
$$
P_h[C_h=h]-\mbox{$\frac{2d^s(d+1)^s}{q}$} \leq P_h[\{C_h=h\}\cap
S_{\sf H}^h] \leq P_h[C_h=h].$$ 
As a consequence, from Corollaries \ref{coro: probability 1
specialization - asymptotic} and \ref{coro: probability of h
specializations - asymptotic} we easily deduce the lemma. The case
$d$ even follows similarly.
\end{proof}

Now we are able to estimate the probability of success of Algorithm
\ref{algo: basic scheme}.
\begin{theorem}\label{th: probability success}
Let $q >2d^s(d+1)^s$ and $s<d$. If $P$ denotes the probability of
failure of Algorithm \ref{algo: basic scheme}, then
\begin{align*}
\big|P-(1-\mu_d)^{h^*}\big| &\le \mbox{$\frac{e^{h^*}}{(d+1)!}$}
+\mathcal{O}\big(\mbox{$\frac{h^*d^s(d+1)^s}{q}
+\frac{(2-\mu_d)^{h^*}\!\!}{q^s}$}\big)\textrm{ for }d\textrm{
odd},\\[1ex]
\big|P-(1-\mu_{d+1})^{h^*}\big| &\le \mbox{$\frac{e^{h^*}}{(d+1)!}$}
+\mathcal{O}\big(\mbox{$\frac{h^*d^s(d+1)^s}{q}
+\frac{(2-\mu_d)^{h^*}\!\!}{q^s}$}\big)\textrm{ for }d\textrm{
even}.
\end{align*}
\end{theorem}
\begin{proof}
Suppose that $d$ is odd. Observe that
$$\sum_{h=1}^{h^*}\mu_d(1-\mu_d)^{h-1}=1-(1-\mu_d)^{h^*}.$$
By Lemma \ref{lemma: prob C_h=h and S_h}, the probability $P^*$
that, for a random choice of $\underline{\bfs a}:=(\bfs
a_1,\ldots,\bfs a_{h^*})\in\FF_{h^*}$ and $\bfs
F_s\in\mathcal{F}_d^s$, there is $h$ with $1\le h\le h^*$ such that
$C(\bfs F,\underline{\bfs a})=h$ and $\bfs F_s(\bfs a_h,-)$
satisfies condition $({\sf H})$ can be estimated as follows:
\begin{align*}
\big|P^*-\big(1-(1-\mu_{d+1})^{h^*}\big)\big|&\le
\sum_{h=1}^{h^*}|P_h[\{C_h=h\}\cap
S_{\sf H}^h]-\mu_d(1-\mu_d)^{h-1}|\\
&\le \sum_{h=1}^{h^*}\Big(\mbox{$\frac{e^{h-1}
+\frac{1}{2}}{(d+1)!}$} +\mbox{$\frac{5
(2-\mu_d)^{h-1}\!\!}{q^s}$}\Big)
+\mbox{$\frac{2h^*(d^s(d+1)^s+1)}{q}$}\\
&\le \mbox{$\frac{e^{h^*}}{(d+1)!}$}+\mbox{$\frac{5}{1-\mu_d}
\frac{(2-\mu_d)^{h^*}\!\!}{q^s}$}
+\mbox{$\frac{2h^*(d^s(d+1)^s+1)}{q}$}.
\end{align*}

It remains to take into account the probability that the routine for
zero-dimensional search is successful. Taking into account that such
a routine has a probability of success of order
$\mathcal{O}(q^{-1})$, applying the Fr\'echet inequalities as in the
proof of Lemma \ref{lemma: prob C_h=h and S_h} we readily deduce the
statement of the theorem. The case $d$ even follows similarly.
\end{proof}

%
%
\bibliographystyle{amsalpha}

\bibliography{refs1, finite_fields}

\end{document}